\documentclass[12pt,a4paper]{amsart}
\usepackage{amsfonts,amscd,amsmath,amssymb,amsthm,mathrsfs,mathtools}
\usepackage{array,latexsym,float,enumerate}
\usepackage[utf8]{inputenc}
\usepackage{graphics,epsfig,xcolor,ifpdf}
\usepackage[top=2cm, bottom=2cm, left=2cm, right=2cm, twoside=false]{geometry}
\usepackage[all]{xy}
\usepackage[normalem]{ulem} 

\numberwithin{equation}{section}

\makeatletter \def\subsection{\@startsection{subsection}{3}%
  \z@{.5\linespacing\@plus.7\linespacing}{.5\linespacing}%
  {\normalfont\itshape}} \makeatother  

\makeatletter \renewenvironment{proof}[1][\proofname]{
  \par\pushQED{\qed}\normalfont
  \topsep6\p@\@plus6\p@\relax
  \trivlist\item[\hskip\labelsep\bfseries#1\@addpunct{.}]
  \ignorespaces}{
  \popQED\endtrivlist\@endpefalse} \makeatother

\theoremstyle{plain}
\newtheorem{thm}{Theorem}[section]

\newtheorem{thmA}{Theorem}[]

\newtheorem{lem}[thm]{Lemma}
\newtheorem{prop}[thm]{Proposition}
\newtheorem{cor}[thm]{Corollary}
\newtheorem{question}[thm]{Question}

\theoremstyle{definition}
\newtheorem{dfn}[thm]{Definition}
\newtheorem{example}[thm]{Example}
\newtheorem{noname}[thm]{}
\newtheorem{subnoname}{}[thm]
\newtheorem{rem}[thm]{Remark}
\newtheorem{construction}[thm]{Construction}
\newtheorem{notation}[thm]{Notation}

\theoremstyle{remark}

\newtheorem*{smallremark}{Remark}

\newtheorem{case}{Case} \makeatletter \@addtoreset{case}{thm}\makeatother
\newtheorem{claim}{Claim} 

\newcommand{\bthm}{\begin{theorem}}
\newcommand{\bprop}{\begin{proposition}}
\newcommand{\blem}{\begin{lemma}}
\newcommand{\bcor}{\begin{corollary}}
\newcommand{\brem}{\begin{remark}}
\newcommand{\bdfn}{\begin{definition}}
\newcommand{\bitem}{\begin{itemize}}
\newcommand{\benum}{\begin{enumerate}}
\newcommand{\bex}{\begin{example}}
\newcommand{\bno}{\begin{noname}}
\newcommand{\bsno}{\begin{subnoname}}
\newcommand{\bsrem}{\begin{smallremark}}
\newcommand{\bnot}{\begin{notation}}
\newcommand{\bcon}{\begin{construction}}
\newcommand{\bca}{\begin{case}}
\newcommand{\bcl}{\begin{claim}}
\newcommand{\beq}{\begin{equation}}
\newcommand{\bpf}{\begin{proof}}
\newcommand{\epf}{\end{proof}}
\newcommand{\eeq}{\end{equation}}
\newcommand{\ecl}{\end{claim}}
\newcommand{\eca}{\end{case}}
\newcommand{\econ}{\end{construction}}
\newcommand{\enot}{\end{notation}}
\newcommand{\esrem}{\end{smallremark}}
\newcommand{\eno}{\end{noname}}
\newcommand{\esno}{\end{subnoname}}
\newcommand{\eex}{\end{example}}
\newcommand{\eitem}{\end{itemize}}
\newcommand{\eenum}{\end{enumerate}}
\newcommand{\ethm}{\end{theorem}}
\newcommand{\eprop}{\end{proposition}}
\newcommand{\elem}{\end{lemma}}
\newcommand{\ecor}{\end{corollary}}
\newcommand{\erem}{\end{remark}}
\newcommand{\edfn}{\end{definition}}
\newcommand{\ble}{\begin{lemma}}
\newcommand{\ele}{\end{lemma}}

\newcommand{\ov}{\overline}

\newcommand{\wt}{\tilde}

\def\8{\infty}
\def\.{\cdot}
\def\A{\mathbb{A}}
\def\An{\mathbb{A}^n}
\def\Ast{\mathbb{A}^1_*}
\def\AuA{\mathbb{A}^1\cup_{\{0\}}\mathbb{A}^1}
\def\P{\mathbb{P}}

\def\C{\mathbb{C}}
\def\Z{\mathbb{Z}}
\def\N{\mathbb{N}}
\def\Q{\mathbb{Q}}

\def\xra{\xrightarrow}

\def\mono{\hookrightarrow}

\def\:{\colon}
\def\map{\dashrightarrow}

\renewcommand{\iff}{\Leftrightarrow}

\def\med{\medskip}

\def\eps{\varepsilon}

\def\Aut{\operatorname{Aut}}
\def\End{\operatorname{End}}

\def\redd{\mathrm{red}}

\def\Pic{\operatorname{Pic}}

\def\Exc{\operatorname{Exc}}

\def\Aut{\operatorname{Aut}}

\def\Pic{\operatorname{Pic}}

\def\Exc{\operatorname{Exc}}

\def\Spec{\operatorname{Spec}}

\def\id{\operatorname{id}}

\def\Et{\operatorname{\acute{E}t}}
\def\Etc{\operatorname{\acute{E}t_\mathbb{C^*}}}
\def\EC{\operatorname{\acute{E}C}}
\def\pr{\operatorname{pr}}

\newcommand{\ti}{\tilde}

\makeatletter \def\@tocline#1#2#3#4#5#6#7{\relax \ifnum #1>\c@tocdepth \else \par \addpenalty\@secpenalty\addvspace{#2} \begingroup \hyphenpenalty\@M \@ifempty{#4}{\@tempdima\csname r@tocindent\number#1\endcsname\relax}{\@tempdima#4\relax} \parindent\z@ \leftskip#3\relax \advance\leftskip\@tempdima\relax \rightskip\@pnumwidth plus4em \parfillskip-\@pnumwidth #5\leavevmode\hskip-\@tempdima \ifcase #1 \or\or \hskip 1em \or \hskip 2em \else \hskip 3em \fi #6\nobreak\relax \dotfill\hbox to\@pnumwidth{\@tocpagenum{#7}}\par \nobreak \endgroup  \fi}
\makeatother

\ifpdf \usepackage[citecolor={green},linkcolor={blue}, colorlinks=true]{hyperref} \else \usepackage[hypertex,citecolor={magenta},linkcolor={blue},colorlinks=true]{hyperref} \fi

\begin{document}



\thanks{Research project partially funded by the Ministry of Science and Higher Education of the Republic of Poland, Iuventus Plus grant no.\ 0382/IP3/2013/7}

\author{Adrien Dubouloz} 
\address{IMB UMR5584, CNRS, Univ. Bourgogne Franche-Comté, F-21000 Dijon, France.}
\email{adrien.dubouloz@u-bourgogne.fr}
\author{Karol Palka}
\address{Institute of Mathematics, Polish Academy of Sciences, Śniadeckich 8, 00-656 Warsaw, Poland} 
\email{palka@impan.pl}

\title[The JC fails for pseudo-planes]{The Jacobian Conjecture\\ fails for pseudo-planes} 
 
\begin{abstract}
A smooth complex variety satisfies the Generalized Jacobian Conjecture if all its étale endomorphisms are proper. We study the conjecture for $\Q$-acyclic surfaces of negative Kodaira dimension. We show that $G$-equivariant counterexamples for infinite group $G$ exist if and only if $G=\C^*$ and we classify them relating them to Belyi-Shabat polynomials. Taking universal covers we get rational simply connected $\C^*$-surfaces of negative Kodaira dimension which admit non-proper $\C^*$-equivariant étale endomorphisms.

We prove also that for every integers $r\geq 1, k\geq 2$ the $\Q$-acyclic rational hyperplane $u(1+u^{r}v)=w^k$, which has fundamental group $\Z_k$ and negative Kodaira dimension, admits families of non-proper étale endomorphisms of arbitrarily high  dimension and degree, whose members remain different after dividing by the action of the automorphism group by left and right  composition. 
\end{abstract}

\maketitle

\section{Main result}
The Jacobian Conjecture asserts that if an algebraic endomorphism of the complex affine space $\An =\Spec (\C[x_1,\ldots,x_n])$  has an invertible Jacobian then it is an isomorphism. The conjecture is open and hard, even for $n=2$. One of the ideas which may contribute to our understanding of it is to study the problem in a broader class of varieties. A general form of the conjecture has been proposed by M. Miyanishi.  

\begin{dfn}[Generalized Jacobian Conjecture]\label{def:GJC} A smooth complex variety satisfies the \emph{Generalized Jacobian Conjecture} if all étale endomorphisms of that variety are proper. 
\end{dfn}

Unless the topological Euler characteristic of the variety is zero, properness is equivalent to the invertibility of the endomorphism (see Lemma \ref{lem:etale-basics}), so for affine spaces the conjecture is equivalent to the original Jacobian Conjecture.  The Generalized Jacobian Conjecture has been studied for certain classes of algebraic varieties, see Miyanishi's lecture notes \cite{Miyanishi-Lectures_on_polynomials} for a good introduction to the subject. It holds for quasi-projective varieties of log general type due to a result of Iitaka (see Lemma \ref{lem:etale-basics}(4)) and it holds for curves by \cite{Miyanishi-etale_endom}. For surfaces there are many positive partial results, see \cite{Miyanishi-etale_endom},  \cite{MiyMa-GJC_and_related_topics}, \cite{Miyan-Kamb_JC}, \cite{MiyMa-hp_with_torus_actions}, \cite{Miyanishi_JC}, \cite{GuMi-GJC_for_singular}. But the most intriguing question, which is in fact the main motivation, is whether the conjecture holds for varieties similar to affine spaces, similar in the sense of having the same logarithmic Kodaira dimension or the same, i.e.\ trivial, rational homology groups. We restrict our attention to the intersection of these two classes in dimension two, that is, to $\Q$-homology planes of negative Kodaira dimension. Similarly to $\A^2$, they all admit an $\A^1$-fibration over $\A^1$ (hence are rational, see Proposition \ref{lem:Qhp_structure}). The only $\Z$-homology plane in this class is the affine plane $\A^2$ (\cite[3.4.3.1]{Miyan-OpenSurf}). As expected, the difficulty of the problem increases as we decrease the complexity of the fundamental group of the surface. Gurjar-Miyanishi \cite[\S 6]{GurjarMiyanishi-GJC_for_Qhp} argued that the conjecture holds if $\pi_1$ is non-cyclic, but later Miyanishi noticed that the proof is wrong in case the fundamental group is $\Z_2*\Z_2$ and he found a counterexample of degree $2$ in that case (see Proposition \ref{prop:GJC_known}). In Section \ref{ssec:Miyanishi} we construct étale endomorphisms of every positive degree on Miyanishi's surface.

The question for $\Q$-homology planes of negative Kodaira dimension which have cyclic fundamental groups, called \emph{pseudo-planes}, remained open. These surfaces are geometrically and topologically the most similar ones to $\A^2$, and hence the Jacobian Conjecture for them is of biggest interest. Our main result is that the conjecture fails for many of them, i.e.\ that their étale endomorphisms do not have to be proper. Surprisingly, this is so even if we impose an additional restriction that we look only for counterexamples invariant with respect to an effective action of an infinite algebraic group. In case a smooth complex variety comes with some effective algebraic action of a group $G$ then we say that it satisfies the \emph{$G$-equivariant Generalized Jacobian Conjecture} if all $G$-equivariant étale endomorphisms of that variety are proper. 

Let $k, r$ be positive integers. The smooth affine surface 
\begin{equation}
\ti S(k,r)=\{x^ry=z^k-1\}\subseteq \Spec (\C[x,y,z])
\end{equation} is simply connected (see Example \ref{ex:tiS(k,r)}) and for every $a\in \{1,2,\ldots,k\}$ coprime with $k$ it admits a free $\Z_k$-action $\eps*_a(x,y,z)=(\eps x, \eps^{-r} y, \eps^{-a}z)$, where we identify $\Z_k$ with the group of $k$-th roots of unity ($\Z_1\cong \{1\}$). The quotient surface, denoted by $S(k,r,a)$, is a pseudo-plane with fundamental group $\Z_k$ and the hyperbolic $\C^*$-action  $\lambda\cdot(x,y,z)=(\lambda x,\lambda^{-r}y,z)$ descends to it. Masuda and Miyanishi \cite{MiyMa-hp_with_torus_actions} characterized the surfaces $S(k,r,a)$, $k,r\geq 2$, as the only pseudo-planes with a hyperbolic $\C^*$-action which admit a unique $\A^1$-fibration. We use this result to show the following reduction.

\begin{thmA}\label{thm:(G,S)=(Cst,S(kra))} Let $S$ be a $\Q$-homology plane of negative Kodaira dimension with an effective action of an infinite algebraic group $G$. The $G$-equivariant Jacobian Conjecture holds for $S$ unless $G=\C^*$ and $S$ is $\C^*$-equivariantly isomorphic to some pseudo-plane $S(k,r,a)$ for some $k,r\geq 2$ and $a\in \{1,2,\ldots,k-1\}$ coprime with $k$.
\end{thmA}

In the exceptional cases we obtain a classification (see Theorem \ref{thm:GJC_fails_v2} for a stronger and more detailed result).
  
\begin{thmA}\label{thm:GJC_fails} Let $k, r\geq 2$ be integers and let $a\in \{1,\ldots,k-1\}$ be coprime with $k$. The $\C^*$-pseudo-plane $S(k,r,a)$ admits a $\C^*$-equivariant étale endomorphism of degree $d$ (hence non-proper for $d>1$) if and only if \begin{equation} d\equiv 1 \mod k(r-1) \text{\ \ \ or\ \ \ \ } k|r,\  a=1 \text{\ and \ } d\equiv r \mod k(r-1).
\end{equation}
In particular, the $\C^*$-equivariant Generalized Jacobian Conjecture fails for every $S(k,r,a)$, $k,r\geq 2$ and hence for every finite cover of it.
\end{thmA}

The counterexamples we construct are closely related to polynomial Belyi maps, called \emph{Shabat polynomials}, which we use to give explicit formulas, see Section \ref{sec:Et_Cst}. Since the conjecture (even in the $\C^*$-equivariant version) fails for the surfaces $\tilde S(k,r)$, $k,r\geq 2$, which are the universal covers of $S(k,r,a)$, by taking a product with $\A^{n-1}$ we obtain the following important result. Up to our knowledge these are the first simply connected rational counterexamples in the literature (cf.\ Remark \ref{rem:first}).

\begin{cor} For every positive integer $n$ there exists a smooth affine variety $\A^n$-fibered over $\A^1$ (hence rational and of negative logarithmic Kodaira dimension) which is simply connected and for which the Generalized Jacobian Conjecture fails.
\end{cor}

\smallskip The counterexamples in Theorem \ref{thm:GJC_fails} are discrete in nature. However, we observe that the situation changes if we do not insist on $\C^*$-equivariance. The following result shows that, although pseudo-planes are arguably the surfaces most similar to the affine plane, the Generalized Jacobian Conjecture for many of them fails very strongly.

\begin{thmA}\label{thm:deformations} Let $\bar r\geq 1, k\geq 2$ and let $S=S(k,k\bar r,1)$ be a pseudo-plane as above. Then for every integer $N\geq 0$ there exist arbitrarily high-dimensional families of non-proper étale endomorphisms of $S$ of degree $k(N(\bar r k-1)+\bar r)$ such that all members of the family are different, even after dividing by the action of the automorphism group of $S$ by left and right composition.
\end{thmA}

We note that the surface $S(k,\bar r k,1)$ is isomorphic to the hypersurface $u(1+u^{\bar r}v)=w^k$ in $\A^3$ and all counter-examples are given by explicit formulas. For instance, for the pseudo-plane $S(2,2,1)\cong \{u(1+uv)=w^2\}$, which has fundamental group $\Z_2$, the formulas for the families from Theorem \ref{thm:deformations} are written in Example \ref{ex:S_2-defromations_explicit}. We hope that our counterexamples could serve as a testing area for other researchers interested in the Jacobian Conjecture.

\medskip 
Finally, we note that boundaries of minimal log smooth completions of our counterexamples (which are rational trees, due to $\Q$-acyclicity) all contain a branching component, i.e.\ they are not chains. The $\Q$-homology planes which admit a completion by a chain of rational curves constitute the most important remaining class to study. By \cite[Theorem 3.4]{MasMiy-Cplus_actions_on_Qhp} and \cite{Gizatullin_quasihom_aff_surf} such surfaces are isomorphic to $S(k,1,a)$ for some $k\geq 1$ and $a\in \{1,\ldots,k-1\}$ coprime with $k$. In particular, Theorem \ref{thm:(G,S)=(Cst,S(kra))} implies that they admit no $G$-equivariant counterexamples for infinite groups $G$. Via Lemma \ref{lem:lifting_eta} it is enough to study the question for their universal covers. 

\begin{question}Does the Jacobian Conjecture hold for $\wt S(k,1)=\{xy=z^k-1\}$, $k\geq 1$?
\end{question}
\noindent Note that $\wt S(k,1)$ is a smooth simply connected surface admitting an $\A^1$-fibration over $\A^1$ (hence rational and of negative logarithmic Kodaira dimension) and that $\wt S(1,1)\cong \A^2$.
 
\tableofcontents 

\medskip
\section{Preliminaries}

We work with complex algebraic varieties.  

\subsection{Branched covers of curves.}\label{rem:Shabat_vs_Gal}

By a curve we mean an irreducible reduced variety of dimension one. Recall that, as a consequence of the Riemann-Hurwitz formula, the Generalized Jacobian Conjecture holds for curves (see \cite[Lemma 1]{Miyanishi-etale_endom}).

\begin{lem}[GJC for curves]\label{prop:GJC(1)_holds} 
Every dominant morphism between curves of equal topological Euler characteristic is finite. In particular, the Generalized Jacobian Conjecture holds for curves.
\end{lem}

Let $\varphi\:X\to Y$ be a finite morphism of smooth curves and let $y\in Y$. If $\varphi^*\{y\}=\sum_{j=1}^k e_j\{x_j\}$, where $e_j\geq e_{j'}$ for every $j\leq j'$ and where all $x_j$ are pairwise distinct, then we call the sequence $\lambda=(e_1,\ldots,e_k)$ the \emph{ramification profile} of $y$. If $y_i$ for $i=1,\ldots, n$ is a numeration of all branching points of $\varphi$ and $\lambda_i$, $i=1,\ldots, n$ are their ramification profiles then we call $(\lambda_1,\ldots,\lambda_n)$ the \emph{ramification profile} of $\varphi$ (over the ordered branching locus $(y_1,\ldots,y_n)$).

\begin{lem}[Branched self-covers of $\A^1$; \cite{Thom_65}]\label{lem:Thom} 
For a positive integer $n$ let $\Lambda=(\lambda_{1},\ldots,\lambda_{n})$, where $\lambda_i=(\lambda_{i,1},\ldots,\lambda_{i,k_i})$, $i=1,\ldots,n$ are non-increasing sequences of positive integers with $\lambda_{i,1}\geq 1$, and let $y_1,\ldots,y_k$ be distinct points in $\A^1$. Then a finite morphism $\varphi\:\A^1\to \A^1$ with ramification profile $\Lambda$ and branching locus $(y_1,\ldots,y_n)$ exists if and only if there exists a positive integer $d$, such that the following conditions hold:
\begin{enumerate}
	\item $\lambda_{i,1}+\ldots+\lambda_{i,k_i}=d$ for every $i=1,\ldots,n$.
	\item $k_1+k_2+\ldots+k_n=(n-1)d+1$.
\end{enumerate}
\end{lem}

The necessity of conditions is easy to see. The first one is satisfied with $d=\deg \varphi$ as a consequence of the finiteness of $\varphi$. Since $e_{top}(\A^1)=1$, the second one follows from the Riemann-Hurwitz formula: $$1=d\cdot 1-\sum_{i,j}(\lambda_{i,j}-1)=d-\sum_i (d-k_i)=\sum_i k_i-(n-1)d.$$

If $X\to \P^1$ is a branched cover with at most three branching points, say $0,1,\8$, then the inverse image of the interval $[0,1]$ cuts $X$ into topologically contractible pieces and has a natural structure of a bi-colored graph with vertices being inverse images of $0,1$ and with adjacent vertices of different colors. A graph of the latter type is called a \emph{dessin on $X$}, see \cite[\S 4]{GirondoGonzalez-Riemann_surf_and_DESSIN} for details.

\begin{prop}[Dessin d'enfant, \cite{GirondoGonzalez-Riemann_surf_and_DESSIN}, Proposition 4.20]\label{thm:dessin_correspondence} Let $X$ be a smooth projective curve. The above assignment induces a bijection between isomorphy classes of branched coverings $X\to \P^1$ with at most three ordered branching points and classes of dessins on $X$ modulo orientation-preserving homeomorphism of $X$ whose restriction induces an isomorphism between the bi-colored graphs.
\end{prop}

We will be especially interested in branched self-covers of $\A^1$ with at most $2$ branching points. In this case $X=\P^1$ and the equivalence of branched covering is simply an automorphism of the source $\P^1$ fixing $\8$, hence an affine function. Such branched self-covers correspond to the so-called \emph{Belyi-Shabat polynomials}, which are polynomials with at most two critical values (values at points in which the differential vanishes). These are extensively studied, in particular  because the absolute Galois group $\operatorname{Gal}(\bar \Q/\Q)$ acts faithfully on them (see \cite[Theorem 4.49]{GirondoGonzalez-Riemann_surf_and_DESSIN}), and hence on planar bi-colored graphs (cf.\ Proposition \ref{thm:dessin_correspondence}). Some of the simplest Belyi-Shabat polynomials are the well-known Chebyshev polynomials of the first and second kind. We call a critical point \emph{non-degenerate} if its ramification index equals $2$.

\begin{example}[Chebyshev polynomials]\label{ex:Chebyshev}
The $n$-th \emph{Chebyshev polynomial of the first kind} $T_n(x)\in\mathbb{C}[x]$ is defined by 
\[ T_n(x)=2xT_{n-1}(x)-T_{n-2}(x) \]
where $T_0=1$ and $T_1(x)=x$.  It is the unique polynomial satisfying $T_n(\cos x)=\cos nx$. It has non-degenerate critical points ($T_n''(x_0)\neq 0$ whenever $T_n'(x_0)=0$), all contained in the interval $(-1,1)$. Its critical values are contained in $\{-1,1\}$. It follows that $T_n(-x)=(-1)^nT_n(x)$ and $T_n(1)=1$. The first few Chebyshev polynomials are $T_2(x)=2x^2-1$, $T_3(x)=4x^3-3x$, $T_4(x)=8x^4-8x^2+1$, $T_5(x)=16x^5-20x^3+5x$.

The rescaled derivatives $U_{n-1}(x):=\frac{1}{n}\frac{d}{dx}T_{n}(x)$, $n\geq 0$ are the \emph{Chebyshev polynomials of the second kind}. They are the unique polynomials satisfying $U_{n-1}(\cos x)=\sin nx/\sin x$. It follows that $U_{n-1}(\pm 1)=(\pm 1)^{n-1}n$. 

Chebyshev polynomials of the first and second kind satisfy the relation
\begin{equation}\label{eq:Tn-Un_square_relation} T_n^2(x)-1=(x^2-1)U_{n-1}^2(x).
\end{equation}
\end{example}

\begin{lem}[Characterization of Chebyshev polynomials]\label{lem:Chebyshev_characterization}
Let $P(x)\in\mathbb{C}[x]$ be a polynomial with non-degenerate critical points and with $\pm 1$ as the only critical values. Assume that $\pm 1$ are not critical points, that $P(1)=1$ and that $P(-1)^2=1$. Then $P$ is a Chebyshev polynomial of the first kind.
\end{lem}

\begin{proof}We treat $P$ as a self-cover of $\P^1=\A^1\cup\{\8\}$. The assumption that the critical points are non-degenerate implies that the bi-colored graph associated with $P$ (see the discussion preceding Proposition \ref{thm:dessin_correspondence}) is a chain. The critical values are $\pm 1$, so changing $P$ to $-P$ if necessary we may assume that there is an orientation-preserving homeomorphism of $(\P^1,\8)$ identifying the bi-colored graphs of $P$ and of the Chebyshev polynomial $T_n$. By Proposition \ref{thm:dessin_correspondence} branched self-coverings of $\P^1$ induced by $\pm P$ and $T_n$ are isomorphic via an isomorphism respecting $\8$, so $\pm P(x)=T_n(ax+b)$ for some $a\in \C^*, b\in \C$. Since $T_n(-x)=(-1)^nT(x)$, we may assume $a>0$. The normalization conditions give $\{-a+b,a+b\}\subseteq \{-1,1\}$, hence $b=0$ and $a=1$. Then $P=\pm T_n$, so in fact $P=T_n$, because $T_n$ and $P$ agree on $x=1$.
\end{proof}

\medskip
\subsection{Rational fibrations and completions.}  

By a \emph{fibration} of an algebraic variety we mean a faithfully flat (hence surjective) morphism whose general fibers are irreducible, reduced and have positive dimension. A \emph{completion} of a fibration $\rho\:S\to B$ of a smooth variety is a triple $(\bar  S,\iota,\bar  \rho)$ consisting of a smooth complete variety $\bar  S$, an embedding $\iota\: S\to \bar  S$ for which $D=\bar  S\setminus S$ is a simple normal crossing divisor and a fibration $\bar  \rho\:\bar  S\to \bar  B$, where $\bar  B$ is a smooth completion of $B$, such that $\bar  \rho| _S=\rho$. In fact we will often refer to $(\bar  S,D,\bar  \rho)$ as the completion, the identification $S\cong \bar  S\setminus D$ being implicitly fixed. A completion is called \emph{minimal} if it does not properly dominate other completions.

\begin{dfn} \label{dfn:fib-preserving} If $\rho_i\:S_i\to B_i$, $i=1,2$ are fibrations of varieties then we say that a morphism $\eta\:S_1\to S_2$ \emph{respects} the fibrations $\rho_1$ and $\rho_2$ if there exists a morphism $\varphi\:B_1\to B_2$ such that $\rho_2\circ\eta=\varphi\circ \rho_1$:
$$\xymatrix{S_1\ar[d]_{\rho_1}\ar[r]^{\eta} & S_2\ar[d]^{\rho_2}\\ B_1\ar[r]^{\varphi} & B_2.}$$
\end{dfn}

A fibration $\rho\:S\to B$ of a normal surface $S$ is called a \emph{$\P^1$-} or \emph{$\A^1$-} or an \emph{$\Ast$-fibration} if its general fibers are isomorphic to $\P^1$, $\A^1=\A^1_{\C}$ or $\Ast=\A^1_{\C}\setminus\{0\}$ respectively. For such a fibration a fiber is called \emph{degenerate} if it is not isomorphic to a general fiber, equivalently it is either non-reduced or its reduced form is not isomorphic to the general fiber. 
Fibers of $\P^1$-fibrations of smooth projective surfaces are well understood. There are no multiple fibers and every reducible fiber contains a $(-1)$-curve which meets at most two other components of the fiber, so iterated contractions of such curves map the fiber onto a smooth $0$-curve. In particular, every such fiber is a tree of rational curves. An $\Ast$-fibration is \emph{untwisted} if its generic fiber is $\A^1_{K,*}$, where $K=\C(B)$, and \emph{twisted} otherwise. 
Fibers of $\A^1$- and $\Ast$-fibrations of affine surfaces can be easily described. By $\mu\A^1\cup_{\{0\}}\nu\A^1$ we denote the scheme $\Spec(\C[x,y]/(x^{\mu}y^{\nu}))$, which is a sum of two multiple affine lines whose reduced forms meet transversally at one point.

\begin{notation} For a curve $C$ contracted by a given fibration $\rho:S\to B$ of a surface we denote by $\mu(C)$ the multiplicity of $C$ in the scheme-theoretic fiber of $\rho$ containing $C$. 
\end{notation}

\begin{lem}[$\A^1$-, $\A^1_*$-fibrations, \cite{Miyan-OpenSurf}, 3.1.4.2, 3.1.7.3]\label{lem:Gamma+Delta}
Let $F$ and $F_{gen}$ denote respectively a degenerate and a general fiber of a fibration of a smooth affine surface.\begin{enumerate}
	\item If $F_{gen}\cong \A^1$ then $F_\redd\cong \bigsqcup_{i=1}^r \A^1$ for some $r\geq 1$.
	\item If $F_{gen}\cong \Ast$ then $F_\redd\cong \Gamma\sqcup \bigsqcup_{i=1}^r \A^1$ for some $r\geq 0$ and $\Gamma\in \{\emptyset$, $\Ast$, $\AuA\}$.  Any two lines constituting a connected component with reduced form $\AuA$ have coprime multiplicities.
\end{enumerate} 
\end{lem}

\begin{proof}The only statement not proven in the above references is the one about multiplicities in (2). Deleting connected components whose reduced form is $\A^1$ we may assume that $r=0$. Let $\bar F$ be the fiber of the minimal completion of the fibration which contains $F$ and let $G_1, G_2$ be the closures of the components of $F$. The boundary $D$ of the completion is connected, because the initial surface is affine. We may assume $\mu(G_1), \mu(G_2)\geq 2$, hence $G_1+G_2$ does not meet sections contained in $D$. The minimality implies that $D\cap \bar F$ contains no $(-1)$-curves, and hence $\bar F$, which is a tree of rational curves, contains a unique $(-1)$-curve, say $G_1$. The difference $\bar F\setminus F$ has exactly two connected components, each meeting one of the sections contained in $D$, hence containing a component of multiplicity $1$. It follows that $\bar F$ is a chain of rational curves. By induction we argue that any two components of $\bar F$ which meet have coprime multiplicities. 
\end{proof}

\smallskip
\subsection{$\Q$-homology planes of negative Kodaira dimension and $\C^*$-pseudo-planes.} \label{ssec:Cst_fibrations}

A $\Q$-\emph{homology plane} is a smooth complex surface whose Betti numbers are the same as those of the affine plane, that is, trivial in positive dimensions. As mentioned in the introduction, every $\Q$-homology plane is affine. These surfaces are important for many problems in affine geometry and the attempts to fully understand them motivated lots of progress (see \cite{Miy-recent_dev}, \cite[\S 3.4]{Miyan-OpenSurf} and \cite{Palka-recent_progress_Qhp} for a review). We need the following more detailed description (see \cite[\S 3.4]{Miyan-OpenSurf} and \cite[5.9, 4.19]{Fujita-noncomplete_surfaces}). By $\kappa$ we denote the logarithmic Kodaira dimension. Cyclic groups are denoted by $\Z_m:=\Z/m\Z$ and the free product of groups by $*$. 

\begin{lem}[$\Q$-homology planes with $\kappa=-\8$]\label{lem:Qhp_structure}
A $\Q$-homology plane $S$ of negative Kodaira dimension has an $\A^1$-fibration $\rho:S\to B$, and for every such fibration the base curve $B$ and the reduced forms of all the fibers are isomorphic to $\A^1$. 

If $\mu_1F_1,\ldots,\mu_nF_n$, with $F_i$ reduced, are all the degenerate fibers of $\rho:S\to B$, then $\pi_1(S)\cong \Z_{\mu_1}*\ldots*\Z_{\mu_n}.$ In particular, $\pi_1(S)$ is finite if and only if it is cyclic.
\end{lem} 

A $\Q$-homology plane is called a \emph{pseudo-plane} if it has negative logarithmic Kodaira dimension and a cyclic fundamental group.  
The following positive result concerning the Generalized Jacobian Conjecture was proved in \cite[\S3, \S6]{GurjarMiyanishi-GJC_for_Qhp}. Some simplifications and a correction of Theorem 6.2 from loc.\ cit.\ (the counterexample in case $\pi_1(S)\cong \Z_2*\Z_2$) were made in \cite[2.4.3(2), 2.3.11]{Miyanishi-Lectures_on_polynomials}.  

\begin{prop}[GJC for $\Q$-homology planes with $\kappa=-\8$] \label{prop:GJC_known} Let $S$ be a $\Q$-homology plane of negative Kodaira dimension with a non-cyclic fundamental group. Then $S$ has a unique $\A^1$-fibration and every étale endomorphism of $S$ respects it. Furthermore, the Generalized Jacobian Conjecture holds for $S$, unless $\pi_1(S)=\Z_2*\Z_2$.
\end{prop}

Recall  that a $\C^*$-action on a smooth surface $S=\Spec (A)$ is called \emph{hyperbolic} if for every fixed point the weights of the induced linear action on the tangent space are non-zero and have different sign. Equivalently (see \cite{Kaup-Fieseler_Cst-actions}), every fixed point is isolated and is not a limit point of every nearby orbit. For such actions, the algebraic quotient morphism $\rho\:S\to S/\C^*=\Spec (A^{\C^*})$ is an $\Ast$-fibration over a smooth curve.  

We now review the description of pseudo-planes which admit hyperbolic $\C^*$-actions. We begin with their universal covers.

\begin{example}[The $\C^*$-surfaces $\ti S(k,r)$]\label{ex:tiS(k,r)}
For fixed integers $k,r\geq 1$, the smooth affine surface \begin{equation}
\ti S(k,r)=\{x^ry=z^k-1\}\subseteq \Spec (\C[x,y,z])
\end{equation} is endowed with an effective hyperbolic $\C^*$-action defined by
\begin{equation}\label{eq:Cst-action_Sti(k,r)}
\lambda\cdot (x,y,z)=(\lambda x,\lambda^{-r}y,z),\  \lambda\in \C^*.
\end{equation}
The algebraic quotient morphism $$\ti \rho=\pr_z\:\ti S(k,r)\to \ti S(k,r)/\C^*=\Spec (\C[z])\cong \A^1$$
is an $\Ast$-fibration with degenerate fibers $\ti \rho^*(\eps^j)\cong \A^1\cup_{\{0\}}r\A^1$, where $\eps$ is a primitive $k$-th root of unity and $j=0,\ldots,k-1$. The degenerate fibers of the minimal completion of $\rho$ into a $\mathbb{P}^1$-fibration are of type $[r,1,(2)_{r-1}]$, that is, they are chains of rational curves with subsequent weights $-r,-1,-2,\ldots,-2$ with exactly $r-1$ weights $(-2)$.

In addition, for every polynomial $P\in \C[x]$, $\ti S(k,r)$ carries a $\C^+$-action defined by 
\begin{equation}\label{eq:Cplus-action_Sti(k,r)}
\Theta^P_t(x,y,z)=(x,y+\frac{(z+tP(x)x^r)^k-z^k}{x^r}, z+tP(x)x^r),\  t\in \C^1,
\end{equation}
whose algebraic quotient morphism $$p=\pr_x\:\ti S(k,r)\to \ti S(k,r)/\C^+=\Spec (\C[x])\cong \A^1$$ is an $\A^1$-fibration trivial over $\A^1\setminus\{0\}$ and with the unique degenerate fiber $p^*(0)=\bigsqcup_{i=1}^k\A^1$. In particular, $\ti S(k,r)$ is a rational surface with logarithmic Kodaira dimension $\kappa(\ti S(k,r))=-\8$, which is simply connected by \cite[5.9, 4.19]{Fujita-noncomplete_surfaces}, as all fibers of $p$ are reduced. Since $\wt S(k,r)$ is affine, $H_2(\wt S(k,r),\Z)$ is free abelian and hence isomorphic to $\Z^{k-1}$, because its topological Euler characteristic is $k$.

Moreover, for every $a\in\{1,\ldots,k\}$ coprime with $k$, $\ti S(k,r)$ admits a free $\C^*$-equivariant $\Z_k$-action defined by 
\begin{equation}
\eps*_a(x,y,z)=(\eps x, \eps^{-r} y, \eps^{-a}z), \label{eq:Zk-action}
\end{equation}
where we identify $\Z_k$ with the group of complex $k$-th roots of unity.
\end{example}

\begin{example}[The $\C^*$-pseudo-planes $S(k,r,a)$]\label{ex:S(k,r,a)}
Let again $k, r$ be positive integers and let $a\in\{1,\ldots,k\}$ be coprime with $k$. Denote by
\begin{equation}\label{eq:quotient_pi}
\pi_a\:\ti S(k,r)\to S(k,r,a)=\ti S(k,r)/\Z_k
\end{equation}
the quotient morphism of $\ti S(k,r)$ by the $\Z_k$-action \eqref{eq:Zk-action}. Note that for $k=1$ we have $a=1$ and $S(1,r,a)=\ti S(1,r)\cong \A^2$. Since it commutes with the $\Z_k$-action, the $\C^*$-action \eqref{eq:Cst-action_Sti(k,r)} descends to an effective hyperbolic $\C^*$-action on $S(k,r,a)$. We have a commutative diagram 
\[\xymatrix{ \ti S(k,r) \ar[d]_{\ti \rho} \ar[r]^{\pi_a} & S(k,r,a) \ar[d]^{\rho} \\ \Spec (\C[z]) \ar[r]^{\pi_a'} & \Spec (\C [t]),}\]  
where $\rho\:S(k,r,a)\to S(k,r,a)/\C^*=\Spec (\C[t])$ is the quotient $\Ast$-fibration and where $\pi_a'(z)=t=1-z^k$. The degenerate fibers of $\rho$ are
\begin{equation}\label{eq:S(k,r,a)_fibers}
F_1=\rho^*(1)=kA_1\cong k \Ast \text{\ \ and\ \ } F_0=\rho^*(0)=A_0\cup rA_2\cong \A^1\cup_{\{0\}}r\A^1.
\end{equation}
The degenerate fibers of the minimal completion of $\rho$ into a $\mathbb{P}^1$-fibration are chains of rational curves, the one containing $F_0$ is of type $[r,1,(2)_{r-1}]$, see Fig.\ \ref{fig:S(k,r,a)_fibers}.

\begin{figure}[h]\centering
\begin{minipage}{3in}\centering \includegraphics[scale=0.6]{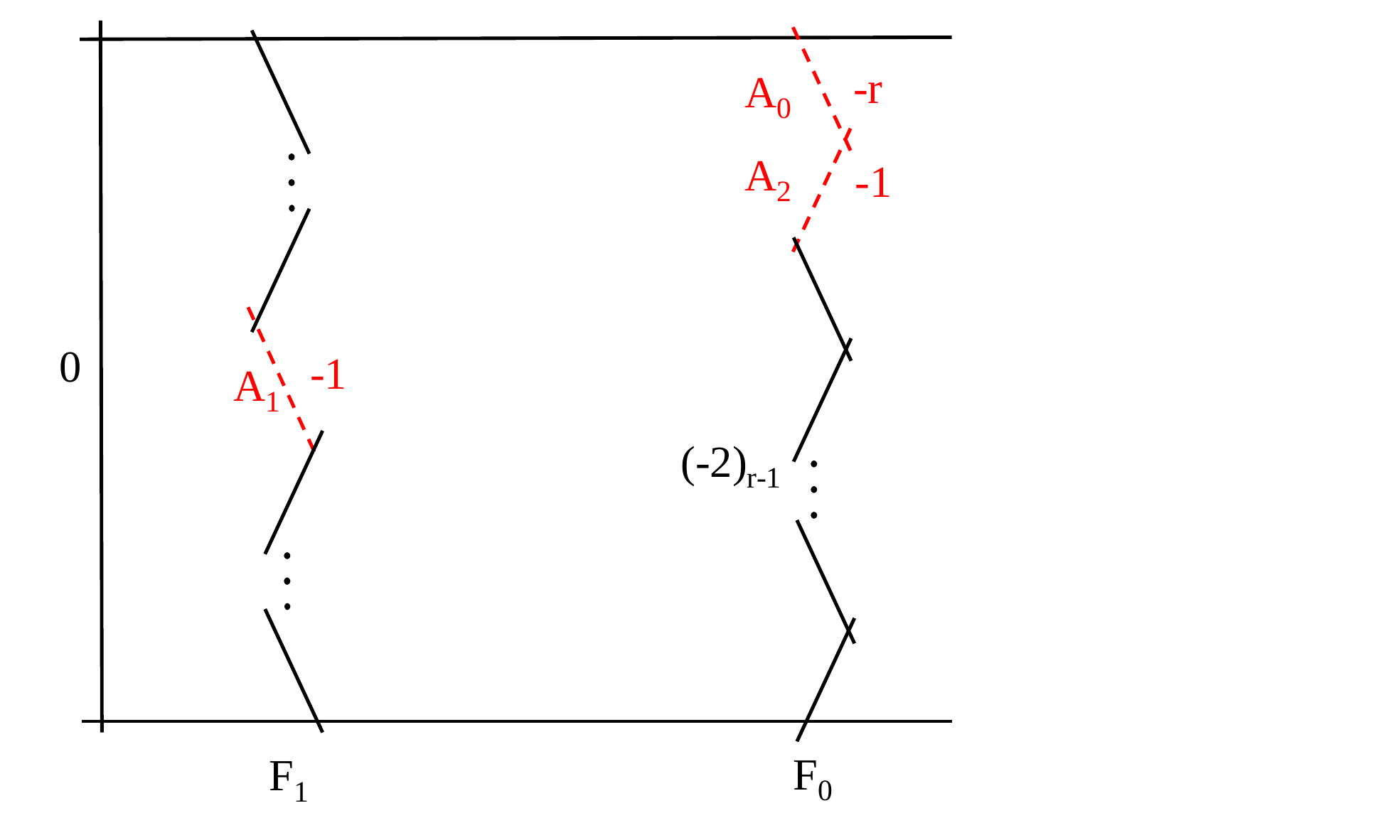}\end{minipage}
\caption{A completion of $\rho\:S(k,r,a)\to \A^1$.}\label{fig:S(k,r,a)_fibers}
\end{figure}

Since $\ti S(k,r)$ is a simply connected surface of negative Kodaira dimension and Euler characteristic equal to $k$, the surface $S(k,r,a)$ is a pseudo-plane with fundamental group $\pi_1(S(k,r,a))\cong \Z_k$ and with a hyperbolic $\C^*$-action. 
\end{example}

\begin{lem}[Characterization of $S(k,r,a)$, \cite{MiyMa-hp_with_torus_actions}, \cite{FlZa-covers_of_Cst-pseudoplanes}] \label{lem:Cst-pseudo-planes} Let $S$ be a non-trivial ($S\not\cong \A^2$) pseudo-plane with a hyperbolic $\C^*$-action. Let $k=|\pi_1(S)|$ and let $r$ be the maximal multiplicity of a fiber component of the quotient morphism isomorphic to $\A^1$. If $r\geq 2$ then there exists $a\in\{1,\ldots,k-1\}$ coprime with $k$ such that $S$ is $\C^*$-equivariantly isomorphic to $S(k,r,a)$ with the $\C^*$-action induced from \eqref{eq:Cst-action_Sti(k,r)}. The isomorphism is unique up to a composition with the group automorphism $\lambda \mapsto \lambda^{-1}$ of $\C^*$. 
\end{lem}

\begin{proof}Since the $\C^*$-action is hyperbolic, the quotient morphism is a $\C^*$-fibration. Since $S\not\cong \A^2$, by Lemma \ref{lem:Qhp_structure} $k\geq 2$. From the above description we see that for $r\geq 2$ there is a smooth relatively minimal completion of the quotient morphism for which the boundary is not a rational chain, so by a theorem of Gizatullin \cite{Gizatullin_quasihom_aff_surf} $S$ has a unique $\A^1$-fibration. Then we infer the result from  \cite{MiyMa-hp_with_torus_actions} and \cite{FlZa-covers_of_Cst-pseudoplanes}.
\end{proof}

\begin{rem}\label{rem:non-hyperbolic_Cst_action} If an effective $\C^*$-action on a $\Q$-homology plane is not hyperbolic then it is a linear action on $\A^2$.
\end{rem}

\begin{proof}In general, if a $\C^*$-action on a smooth affine variety $S$ has an elliptic or a parabolic fixed point then by \cite[Proposition 1.9]{Rynes-smooth_affine_Cst-surfaces} (which easily follows from \cite{Bass_Haboush-linearizing_group_actions}) $S$ is respectively either equivariantly isomorphic to some affine space with a linear $\C^*$-action or it is a $\C^*$-vector bundle over the smooth quotient $S/\C^*$. In case $S$ is a $\Q$-homology plane the latter quotient is $\Q$-acyclic, hence isomorphic to $\A^1$, so in both cases $S\cong \A^2$ and the $\C^*$-action is linear.
\end{proof}

\medskip
\subsection{\'{E}tale endomorphisms respecting fibrations.}\label{ssec:Etale_endo}

Recall that a morphism $f\:X\to Y$ between varieties is called étale if it is flat and unramified. Since it is flat, it is open and equidimensional (all nonempty fibers of $f$ have the same dimension) and since it is also unramified, it is quasi-finite (its fibers are finite). Assume that $X$ and $Y$ are smooth. If $f\:X\to Y$ is dominant equidimensional then by \cite[23.1]{Matsumura_ring_theory} it is automatically flat. It follows from \cite[Cor.\ 2, \S III.5]{Mumford-red_book} that $f$ is étale if and only if it is a local isomorphism in the complex topology, equivalently, if and only if the differential is an isomorphism for every point of $X$.

For a quasi-finite morphism of smooth quasi-projective varieties $f\:X\to Y$, the Zariski-Nagata purity theorem implies that the complement $E\subset X$ of the étale locus of $f$ is of pure codimension one. It supports an effective, canonically defined \emph{ramification divisor} of $f$, denoted by $R_{f}$, which is linearly equivalent to $f^*K_Y-K_X$, see \cite[pp.\ 202]{Iitaka_AG}. The image of the ramification locus is called the \emph{branching locus}. 

\begin{lem}[\'Etale endomorphisms]\label{lem:etale-basics}
Let $S$ be a smooth variety and let $\eta\:S\to S$ be étale.
\begin{enumerate}
	\item For every reduced divisor $D$ on $S$ the divisor $\eta^*D$ is reduced. 
	\item If $\deg \eta=1$ then $\eta$ is invertible.
	\item If $\eta$ is proper and $e_{top}(S)\neq 0$ then $\eta$ is invertible.
	\item Let $(\ov S,D)$ be a smooth completion of $S$ and let $\Phi_m\:\ov S\map \P_m:=\P(|m(K_{\ov S}+D)|)$ be the $m$-th log canonical map. Then $\Phi_m\circ \eta=p_m(\eta)\circ \Phi_m$ for some $p_m(\eta)\in \Aut \P_m$.
	\item A restriction of $\eta$ to a closed subvariety of $S$ is unramified. In particular, if $\ell\subseteq S$ is isomorphic to $\A^1$ then $\eta(\ell)$ has no unibranched singular points (no cusps) and $\eta_{|\ell}$ has degree $1$. 
\end{enumerate}
\end{lem}

\begin{proof} 
(1) This follows from the fact that $\eta$ is unramified, see e.g. \cite[Proposition I.3.2 and Definition I.3.2]{SGA1}.

(2) If $\deg \eta=1$ then $\eta$ is an open embedding, hence an isomorphism by the Ax-Grothendieck theorem \cite[Theorem 10.4.11]{EGA-IV-3}.

(3) Since $\eta$ is étale, it has finite fibers, so being proper, it is finite by \cite[Theorem 8.11.1]{EGA-IV-3}. Finite étale morphisms are covers in the complex topology, so we get $e_{top}(S)=\deg \eta\cdot e_{top}(S)$, and hence $\deg \eta=1$.  By (2) $\eta$ is invertible. 

(4) See \cite[Propositon 11.9]{Iitaka_AG}.

(5) The first statement follows from the fact that a closed immersion is unramified and that the composition of two unramified morphisms is unramified. Assume $\ell\subseteq S$ is isomorphic to $\A^1$. Let $\eta'\:\ell \to \ell'$ be the lift of $\eta|_{\ell}$ to the normalization $\ell'$ of $\eta(\ell)$. Then $\kappa(\ell')\leq \kappa(\ell)$, so $\ell'$ is isomorphic to $\A^1$ or $\P^1$. But $\eta'$ is unramified, so it extends to a, necessarily finite, endomorphism of $\P^1$ with at most one ramification point. By the Riemann-Hurwitz formula this endomorphism has degree $1$, so $\eta|_{\ell}$ has degree $1$. 
\end{proof}

Given a variety $S$ we denote the monoid of its étale endomorphisms by $\Et(S)$.

\begin{lem}[Lift to the universal cover]\label{lem:lifting_eta} 
Let $\pi\:\ti S\to S$ be a finite étale cover of smooth varieties. Assume $\ti S$ is simply connected. Then every $\eta\in \Et(S)$ lifts to some $\ti\eta\in \Et(\ti S)$ such that $\pi\circ\ti \eta=\eta\circ \pi$. In particular, $\deg \ti \eta=\deg \eta$, so $\eta\in \Aut(S)$ if and only if $\ti \eta \in \Aut(\ti S)$.
\end{lem}

\begin{proof} First recall that $\pi$ has the universal property that it factors through every finite étale cover $f\:X\to S$ of $S$. Indeed the existence of a morphism $\pi_X\:\ti S\to X$ such that $\pi=f\circ \pi_X$ is equivalent to the existence of a section of the projection $\ti S\times_S X\to \ti S$. But the latter, being a pullback of $f$, is a finite étale cover of $\ti S$, so since $\ti S$ is simply connected, it is a trivial cover, which therefore admits a section. Let now $(S',\eta',\pi')$ be the fiber product of $\eta$ and $\pi$. Then  $\eta'\:S'\to \ti S$ is étale of degree $\deg \eta'=\deg \eta$ while  $\pi'\:S'\to S$ is étale and finite, of degree $\deg \pi'=\deg \pi$. By the universal property of $\pi$, there exists a morphism $\alpha\:\ti S\to S'$ such that $\pi'\circ \alpha=\pi$, and then $\ti \eta=\eta'\circ \alpha$ is as required.
\end{proof}

\begin{notation}\label{dfn:Et(S,rho)}
If $\rho\: S\to B$ is a fibration of algebraic varieties and an étale endomorphism $\eta$ of $S$ respects $\rho$ then by Definition \ref{dfn:fib-preserving}, there exists an endomorphism $\eta_\rho\:B\to B$, such that $\rho\circ \eta=\eta_\rho\circ \rho$. We denote by $\Et(S,\rho)$ the submonoid of $\Et(S)$ consisting of étale endomorphisms respecting $\rho$.
\end{notation}

The following Lemma is a minor improvement of  \cite[3.1]{GurjarMiyanishi-GJC_for_Qhp}.

\begin{lem}[Condition on multiplicities for étale endomorphisms]\label{lem:multiplicities}
Let $\rho_i\:S_i\to B_i$, $i=1,2$ be fibrations of smooth surfaces over smooth curves and let $\eta\: S_1\to S_2$ be an endomorphism respecting $\rho_1$ and $\rho_2$, with the induced morphism $\varphi\: B_1 \to B_2$. Let $e(\varphi,p)$ denote the ramification index of $\varphi$ at $p\in B_1$. If $\eta$ is étale then  for every irreducible curve $C$ contracted by $\rho_1$ we have 
\begin{equation}\mu(\eta(C))=e(\varphi,\rho(C))\cdot \mu(C). \label{eq:multiplicities}
\end{equation} Conversely, if \eqref{eq:multiplicities} holds for every irreducible curve $C$ contracted by $\rho_1$ and $\eta$ is étale on general fibers of $\rho_1$ then $\eta$ is étale.
\end{lem}

\begin{proof}Let $C_1$ be an irreducible curve contracted by $\rho_1$. Put $C_2=\eta(C_1)$, $p_i=\rho_i(C_i)$ and $e=e(\varphi,p_1)$. Write $\eta^*(C_2)=mC_1+R$ for some $m\geq 1$ and an effective divisor $R$ contained in fibers and not containing $C_1$. Since $\eta^*\circ\rho_2^*(p_2)=\rho_1^*\circ\varphi^*(p_2)$, looking at the coefficient of $C_1$ in the latter divisor we get
\begin{equation}\mu(C_2)m=e\mu(C_1). \end{equation}

If $\eta$ is étale then $m=1$ by \ref{lem:etale-basics}(1), which gives  \eqref{eq:multiplicities}. For the converse implication, we note that since $\eta$ is quasi-finite, it is flat by \cite[23.1]{Matsumura_ring_theory}. Because $S_1$ is smooth, the Zariski-Nagata purity theorem implies that the ramification locus of $\eta$ is either empty or has pure codimension $1$. But by \eqref{eq:multiplicities} $m=1$ for every $C_1$ as above, so the ramification divisor has no components contained in fibers. By assumption $\eta$ is étale along general fibers of $\rho_1$, hence the ramification divisor of $\eta$ is trivial.
 \end{proof}

\begin{example} Let $B$, $C$ be smooth curves and let $\eta_C$ be a dominating endomorphism of $C$. Then $\eta=\id_B\times \eta_C$ is an endomorphism of $(B\times C,\pr_B)$ which trivially satisfies \eqref{eq:multiplicities} and which is étale if and only if $\eta_C$ is étale. Note that if $C\neq \A^1, \P^1$ then $\kappa(C)\geq 0$, so by \cite[Theorem 11.7]{Iitaka_AG} every dominating endomorphism of $C$ is automatically étale. On the other hand, dominating endomorphisms of degee at least $2$ are never étale for $C=\A^1, \P^1$.
\end{example}

\begin{lem}[One non-reduced fiber]\label{lem:when_etaB_has_deg=1} 
Let $S$ be a smooth surface equipped with a fibration $\rho\:S\to B$ with at most one non-reduced fiber. Then every $\eta\in \Et(S,\rho)$ induces a (finite) étale endomorphism of $B$.
\end{lem}

\begin{proof} By Lemma \ref{prop:GJC(1)_holds} $\eta_\rho$ is finite. Suppose that $\eta_\rho\:B\to B$ has ramification index $d>1$ at some point $p\in B$. By \cite[Prop. 11.7]{Iitaka_AG} $B\cong \A^1$ or $\P^1$. Let $q=\eta_\rho(p)$ and for any $u\in B$ put $F_u=\rho^*(u)$. By definition $\eta_\rho^*(q)-dp$ is effective, hence $\rho^*(\eta_\rho^*(q)-dp) =\eta^*(\rho^*(q))-d\rho^*(p) =\eta^*(F_q)-d F_p$ is effective, so $\eta^*(F_q)$ is a non-reduced divisor and, because $\eta$ is étale, the fiber $F_q$ is non-reduced by Lemma \ref{lem:etale-basics}(1). So if $\rho$ has no multiple fiber then we are done. Otherwise, $F_q$ must be the unique non-reduced fiber of $\rho$. Since $\eta_\rho$ is finite, the restriction $B\setminus \eta_\rho^{-1}(q)\to B\setminus \{q\}$ is finite and étale, hence  $e_{top}(B\setminus \eta^{-1}_\rho(q))=\deg \eta_\rho\cdot e_{top}(B\setminus \{q\})$. Since $\deg \eta_\rho\geq d>1$, the latter is impossible for $B=\P^1$, hence $B=\A^1$ and $e_{top}(\eta_\rho^{-1}(q))=0$, so $\eta_\rho^{-1}(q)=\{p\}$, $\deg \eta_\rho=d$ and $\eta^*(F_q)=\deg \eta_\rho\cdot  F_p$. Thus $\eta_\rho$ is a cyclic cover totally ramified over $q$. Since $F_q$ contains a non-reduced component, by Lemma \ref{lem:etale-basics}(1) the fiber over $\eta_\rho(q)$ contains one too, so $\eta_\rho(q)=q$ and hence $p=q$. We get $\eta^*(F_q)=\deg \eta_\rho\cdot  F_q$. Replacing $\eta$ with some $\eta^k$, $k\geq 1$ we may assume that $\eta$ maps some irreducible component $A$ of $F_q$ to itself. We get $\eta^*(\mu(A)A)=(\deg \eta_\rho)\cdot \mu(A)A$, hence $\eta^*(A)=\deg \eta_\rho\cdot A$. But on the other hand, since $\eta$ is étale, $\eta^*A$ is reduced by Lemma \ref{lem:etale-basics} (1); a contradiction. Thus $\eta_\rho$ is étale.
\end{proof}

\begin{cor}\label{cor:Et(S,p)-for_A1-fib_pseudo-planes} If an étale endomorphism of a pseudo-plane respects some $\A^1$-fibration then it is an automorphism.
\end{cor}

\begin{proof} Let $\rho\:S\to B$ be an $\A^1$-fibration of a pseudo-plane $S$ and let $\eta\in\Et(S,\rho)$. By Lemma \ref{lem:etale-basics}(5) $\deg \eta=\deg \eta_\rho$. By Lemma \ref{lem:Qhp_structure} $p$ has at most one degenerate fiber, so Lemma \ref{lem:when_etaB_has_deg=1} says that $\eta_\rho$ is an automorphism, and hence $\eta$ is an automorphism by Lemma \ref{lem:etale-basics}(2).\end{proof}

Let $\rho\:S\to B$ be a fibration of some smooth surface $S$ onto a smooth curve $B$ and let $\eta\:S\to S$ be an endomorphism respecting $\rho$. Denote by $\eta_\rho$ the endomorphism induced on the base and by $(S',\rho',\eta_\rho')$ the normalized fiber product of $\rho$ and $\eta_\rho\:B\to B$. Since $S$ is smooth, hence normal, by the universal properties of the fiber product and of the normalization there exists a unique morphism $j\:S\to S'$ such that the following diagram commutes:
\begin{equation}\label{diagram:factorization}
\xymatrix@C=2.5pc{S \ar[dr]_{\rho} \ar@/^2pc/[rr]^\eta \ar@{.>}[r]^{\exists ! j} & S' \ar[d]^{\rho'}  \ar[r]^{\eta_\rho'} & S \ar[d]^{\rho} \\ & B \ar[r]^{\eta_\rho} & B. }
\end{equation}

As we will see now, in many cases $j$ is an open embedding.

\begin{lem}\label{lem:factorization_sigma=id}
Let $\rho\:S\to B$ be a fibration of a smooth affine surface over a smooth curve $B$. Assume that $\eta\in \Et(S,\rho)$ and one of the following holds:
\begin{enumerate}
\item general fibers of $\rho$ are not isomorphic to $\Ast$,
\item $\rho$ is the quotient $\Ast$-fibration of some $\C^*$-action on $S$ or
\item $\rho$ is an untwisted $\Ast$-fibration with at least one fiber having a reducible connected component.
\end{enumerate}
Then, with the notation as above, $\eta_\rho'$ is finite and $j$ is an open embedding, both $\C^*$-equivariant in case (2). In particular $\deg \eta=\deg \eta_\rho$.
\end{lem}

\begin{proof}
By Lemma \ref{prop:GJC(1)_holds} $\eta_\rho$ is finite, so since the normalization morphism is finite, $\eta_\rho'$ is finite. Since $\eta$ is quasi-finite, so is $j$. By the Zariski Main Theorem (\cite[Theorem 8.12.6]{EGA-IV-3}) we have a decomposition $j=\sigma\circ \wt j$, where $\wt j\:S\to S^\dagger$ is an open embedding and $\sigma\:S^\dagger\to S'$ is finite. Note that $j$ is an open embedding if and only if $\sigma$ is an isomorphism. Since $S'$ is normal, the latter holds if and only if $\deg \sigma=1$.

(1) Let $F_{gen}$ denote a general fiber. By the Riemann-Hurwitz formula $(\deg \sigma-1)e_{top}(F_{gen})\geq 0$, so we may assume $e_{top}(F_{gen})\geq 0$. Since $F_{gen}\not\cong \Ast$ we have $F_{gen}\cong \A^1$. Then $\sigma$ is an isomorphism by Lemma \ref{lem:etale-basics}(5). We may further assume that $\rho$ is an untwisted $\Ast$-fibration.

(2) By the universal property of the normalization the $\C^*$-action on the fiber product of $\eta_\rho$ and $\rho$ lifts to a $\C^*$-action on $S'$, for which $\rho'$ is the quotient morphism. Since $\rho$ is a trivial principal homogeneous $\C^*$-bundle over some non-empty Zariski open subset $B_0\subseteq B$, taking $B'=B_0\cap \eta_\rho^{-1}(B_0)$ the restriction of $\eta$ to $\rho^{-1}(B')\cong B'\times \Ast$ can be written as $(b,t)\mapsto (\eta_\rho(b),f(b)t^k)$ for some $f\in \C(B')^*$ and $k=\pm \deg j$. Due to the $\C^*$-equivariance of $\eta$, for every $\lambda\in\C^*$ we get $f(b)(\lambda t)^k=\lambda f(b) t^k$, hence $\lambda^k=\lambda$. Thus $k=1$ and hence $j$ is an open immersion. Then the restriction of $j$ to $\rho^{-1}(B')$ is $(b,t)\mapsto (b,f(b)t)$, so $j$, and hence $\eta_\rho'$, is $\C^*$-equivariant.

(3) Let $F_1$ be a reducible connected component of some fiber $\rho^*(p_1)$. By Lemma \ref{lem:Gamma+Delta}(2) $F_1\cong \mu_1\A^1\cup_{\{0\}}\mu_2\A^1$, where $\mu_1,\mu_2$ are coprime positive integers and, since $\eta$ is étale, $F_2=\eta(F_1)\cong \mu_1'\A^1\cup_{\{0\}}\mu_2'\A^1$. By Lemma \ref{lem:multiplicities} $\mu_i'=e(\eta_\rho,p_1)\mu_i$ for $i=1,2$. By Lemma \ref{lem:Gamma+Delta}(2) $e(\eta_\rho,p_1)=1$. Then by \cite[Lemma 2.4.1(3)]{Miyanishi-Lectures_on_polynomials} $\sigma$ is a cyclic Galois cover and $\wt j(F_1)$ is invariant with respect to this action. The action on $\wt j(F_1)$ is free, as the induced morphism $\wt j(F_1)\to F_2$ is étale. But the intersection point of the two components of $\wt j(F_1)$ is necessarily a fixed point, hence the Galois group is trivial, from which it follows that $\sigma$ is an isomorphism.
\end{proof}

\section{Proof of Theorem \ref{thm:(G,S)=(Cst,S(kra))}. Reduction to $\C^*$-actions.}

In this section we prove Theorem \ref{thm:(G,S)=(Cst,S(kra))} in case $G=\C^+$, that is, we prove the $\C^+$-equivariant Jacobian Conjecture for $\Q$-homology planes of negative Kodaira dimension.

\begin{notation}
If $G$ is an algebraic group and $X$ is a $G$-variety then we denote by $\Et_G(S)$ the monoid of $G$-equivariant étale endomorphisms of $S$.    
\end{notation}

\subsection{Non-proper étale endomorphisms respecting an $\A^1$-fibration.} \label{ssec:Cplus-equivariance}

If $\C^+$ acts effectively on a normal affine surface $S$ then the algebraic quotient morphism $\rho\:S\to B$ is an $\A^1$-fibration, hence the surface has negative Kodaira dimension, and any $\C^+$-equivariant étale endomorphism respects this fibration. We are therefore led to the study of étale endomorphisms of $\Q$-homology planes respecting $\A^1$-fibrations. 

Let $\rho\:S\to B$ be an $\A^1$-fibration of a $\Q$-homology plane. By Lemma \ref{lem:Gamma+Delta} $B\cong \A^1$. By Proposition \ref{prop:GJC_known} and Corollary \ref{cor:Et(S,p)-for_A1-fib_pseudo-planes} if $\pi_1(S)\ncong \Z_2*\Z_2$ then $\Et(S,\rho)=\Aut(S,\rho)$, hence the $\C^+$-equivariant Generalized Jacobian Conjecture holds for $S$. The case $\pi_1(S)\cong \Z_2*\Z_2$ requires further attention. As before, let $T_n$ and $U_{n}$ denote the Chebyshev polynomials of degree $n$ of the first and second kind respectively.

\begin{prop}\label{prop:descent_for_A1-fibrations}
Let $\rho\:S\to \A^1$ be an $\A^1$-fibration of a $\Q$-homology plane. Assume $\eta\in\Et(S,\rho)$ has degree $n>1$. Then $\pi_1(S)\cong \Z_2*\Z_2$ and there exist
\begin{enumerate}[(a)]
\item a birational morphism $\sigma\:S\to \A^2=\Spec (\C[x,y])$ restricting to an isomorphism off the fibers $\rho^*(\pm 1)$ such that $\rho={\pr_1}\circ \sigma$ and
\item polynomials $a,b\in \C[x]$ such that $a(1)=\pm 1/n$ and $a$ vanishes only on (some) zeros of $U_{n-1}$
\end{enumerate}
such that for the endomorphism $\eta_0\in \End(\A^2,\pr_1)$ given by 
\begin{equation}\label{eq:eta_0}
\eta_0(x,y)=(T_n(x),U^2_{n-1}(x)a(x)y+(x^2-1)U_{n-1}(x)b(x)),
\end{equation} 
the following diagram commutes
\[\xymatrix{S\ar[d]_{\eta}\ar[r]^{\sigma} & \A^2\ar[d]^{\eta_0}\ar[d]^{\eta_0}\ar[r]^{\pr_1} & \A^1\ar[d]^{T_n}\\S\ar[r]_{\sigma} & \A^2\ar[r]_{\pr_1} & \A^1}\]
for some choice of coordinates on $\A^1$.
\end{prop}

\begin{proof}By Proposition \ref{prop:GJC_known} $\pi_1(S)$ is isomorphic to $\Z_2*\Z_2$ and $\Et(S)=\Et(S,\rho)$. By Lemma \ref{lem:Qhp_structure}, $\rho$ has exactly two degenerate fibers, say over $x=\pm 1$, both isomorphic to $2\A^1$. Since the general fiber is $\A^1$, we have $\deg \eta_\rho=\deg \eta\geq 2$. By Lemma \ref{lem:multiplicities}, $\eta_\rho$ has non-degenerate critical points (the ramification indices equal $2$) and $\eta_\rho\{-1,1\}\subseteq \{-1,1\}$, so renaming $\pm 1$ if necessary we may assume $\eta_\rho(1)=1$ and $\eta_\rho(-1)=\pm 1$. So $\eta_\rho=T_n$ by Lemma \ref{lem:Chebyshev_characterization} with $n=\deg \eta$.

Let $\bar \rho\:\ov S\to \P^1$ be a minimal smooth completion of $\rho$ and let $D_h$ be the section of $\bar \rho$ contained in $D=\ov S\setminus S$. Denote by $\bar F_{\pm 1}$ the fiber of $\bar \rho$ over $\pm 1$. Since $\bar F_{\pm 1}\cap S$ has multiplicity two, $\bar F_{\pm 1}$ has dual graph 
$$\xymatrix@C=1.5pc@R=1.5pc{{-1}\ar@{-}[r] & {-2}\ar@{-}[r] & {\ldots}\ar@{-}[r] & {(-2)}\ar@{-}[r]\ar@{-}[d] & {-2}\ar@{.}[r] & {\bullet} \\ {} & {} & {} & {-2} & {} & { }}$$
where the black dot represents $D_h$. Let $k+1$ be the maximum of the number of components of $\bar F_{\pm 1}$. Contract successively $(-1)$-curves in the fiber with biggest number of components until the induced fibers over $\pm 1$ have the same number of components. Then continue with simultaneous contractions of $(-1)$-curves - each time one $(-1)$-curve over each of $\pm 1$. This gives a sequence of contractions 
\begin{equation*}
\ov S=\ov S_k\xra{\sigma_{k-1}}\ldots\xra{\sigma_1}\ov S_1\xra{\sigma_0}\ov S_0.
\end{equation*} In particular, $\ov S_0$ is a $\P^1$-bundle over $\P^1$ and $\ov S_1$ and $\ov S_2$ have Picard ranks $4$ and $6$, respectively. By choosing $\Exc \sigma_0$ correctly we may, and will, assume that the components of $\bar F_{\pm 1}$ meeting $D_h$ are not contracted in this process. We order the components of both $\bar F_{\pm 1}$ by increasing multiplicity in the fiber (we assume the ones meeting $D_h$ are the first ones). Clearly, components which are older in this order are contracted first in the sequence above.

Let $S\xra{j}S'\xra{\eta_\rho'}S$ be the factorization \eqref{diagram:factorization} and let $\bar \rho'\:\bar S'\to\P^1$ be a minimal normal completion of the induced $\A^1$-fibration of $S'$, smooth along $\bar S'\setminus S'$. Since $\eta_\rho$ is étale at $\pm 1$, $\eta_\rho'$ is a local analytic isomorphism over $\pm 1$, so $\eta_\rho'$ has no base points on the fibers $\bar F'_{\pm 1}=(\bar \rho')^{-1}(\pm 1)$ and maps them isomorphically onto their images $\bar F_{\pm 1}\subseteq \ov S$. But $j$ is an open embedding, so it has no base point either. It follows that $\eta$ has no base points on $\bar F_{\pm 1}$.

For $i\geq 2$ let $D_i$ be the direct image of $D$ on $\ov S_i$ with the last components of the direct images of $\bar F_{\pm 1}$ (in the order defined above) deleted and let $D_1$, $D_0$ be the direct images of $D_h$ on $\ov S_1$ and $\ov S_0$, respectively. For $i\geq 0$ put $V_i=\ov S_i\setminus D_i$. All divisors $D_i$ for $i\neq 2$ are connected. The divisor $D_2\cap \bar F_{\pm 1}$ has $3$ connected (irreducible) components. All $V_i$ for $i\geq 0$ are quasi-affine surfaces with an induced $\A^1$-fibration over $\A^1$ and with degenerate fibers over $\pm 1$ only. For $i\geq 3$ the fibers are isomorphic to $2\A^1$. For $i=2,1,0$ they are isomorphic to $2\A^1_*$ and $\AuA$ and $\A^1$, respectively. The restriction to $V_0$ of the $\P^1$-bundle extending $\rho$ is trivial, hence it can be written as the projection $\rho_0\:\Spec (\C[x,y])\to \Spec (\C[x])$. In particular, $V_0\cong \A^2$. 

Let $\eta_i$ denote $\eta$ treated as a rational endomorphism of $V_{i}$. It has no base points off the fibers over $x=\pm 1$. As we have seen, $\eta$ is well defined on $\bar F_{\pm 1}$. Since the divisor $D\cap \bar F_{\pm1}$ contains no $(-1)$-curves, $\eta$ maps $\bar F_{\pm 1}$ isomorphically to $\bar F_{(\pm 1)^n}$. Because the extended dual graph on the figure above has no symmetry, $\eta$ respects the order of components we defined. It follows by a descending induction on $i$ that $\eta_i$ maps the fibers over $x=\pm 1$ isomorphically to their images. Therefore, $\eta_i$ is an endomorphism of $V_i$. We obtain a commutative diagram
\begin{equation}\label{eq:S_A^1-fibered_descent}
\xymatrix{S\ar[d]_{\eta}\ar[r]^{\sigma_{k-1}} & V_{k-1}\ar[d]^{\eta_{k-1}}\ar[r] &\ldots\ar[r]^{\sigma_1} & V_1\ar[d]^{\eta_1}\ar[r]^{\sigma_0}  & V_0\ar[d]^{\eta_0}\ar[d]^{\eta_0}\ar[r]^{\rho_0} & \A^1\ar[d]^{\eta_\rho=T_n}\\
S\ar[r]^{\sigma_{k-1}} &V_{k-1}\ar[r] & \ldots\ar[r]^{\sigma_1} & V_1\ar[r]^{\sigma_0}  & V_0\ar[r]^{\rho_0} & \A^1}
\end{equation}
from which $\eta_0(x,y)=(T_n(x),A(x)y+B(x))$ for some $A\in \C[x]\setminus 0$ and $B\in \C[x]$. 

We now review the conditions imposed on $A$ and $B$ by the fact that $\eta_0$ lifts to $\eta$. Let $\sigma=\sigma_0\circ\cdots\circ\sigma_{k-1}$. Up to changing the coordinate $y$ by $y+c_1x+c_0$ for some $c_0,c_1\in \C$, we may assume that $\sigma$ contracts the fibers $\rho^{-1}(\pm1)$  onto the points $(\pm1,0)\in V_0$. The commutativity of the above diagram implies that $\eta_0$ maps the base points of $\sigma$, including  infinitely near ones, to the base points of $\sigma$. In particular, $\eta_0(\pm1,0)=((\pm 1)^n,0)$, hence $B(1)=B(-1)=0$, so $B=\left(x^2-1\right)B_1$ for some $B_1\in \C[x]$. Also, since $\eta$ is quasi-finite, we see that $\eta_0$ contracts a fiber over $x_0$ if and only if $x_0\neq \pm 1$ and $T_n(x_0)=\pm 1$, equivalently if and only if $U_{n-1}(x_0)=0$ (see \eqref{eq:Tn-Un_square_relation}), and in each case the image is $(\pm 1,0)$. Since $U_{n-1}$ is separable, $A=U_{n-1}A_1$ and $B_1=U_{n-1}b$ for some $A_1,b\in \C[x]$. 

The morphism $\sigma_0\circ \sigma_1\:V_2\to V_0$ is a restriction of the blowup of the ideal $(x^2-1,y^2)$. Let $E_{1,x}\subseteq V_1$ and $E_{2,x}\subseteq V_2$ be the exceptional divisors of $\sigma_0$ and $\sigma_1$ over $x=\pm 1$. The functions $y$ and $v_2=\frac{x^2-1}{y^2}$ are regular on an open subset of $V_2$ containing $E_{2,\pm 1}$ and in the coordinates $(v_2,y)$ the latter is described by $y=0$ (and $x=\pm 1$). We have $\eta_2^*(y)=U_{n-1}\cdot (A_1y+(x^2-1)b)$ and, by \eqref{eq:Tn-Un_square_relation}, $\eta_2^*(x^2-1)=T_n^2-1=(x^2-1)U^2_{n-1}$, so
\begin{equation}\label{eq:u}
\eta_2^*v_2=\frac{x^2-1}{(A_1y+(x^2-1)b)^2}\ .
\end{equation}

Denote by $F_x$ the reduced form of the fiber of $\rho_2=\rho_0\circ \sigma_0\circ \sigma_1\:V_2\to \A^1$ over $x$ and by $p_{\pm 1}\in E_{2,\pm 1}=F_{\pm 1}\cong \Ast$ the center of $\sigma_3$. Let $\lambda_{\pm 1}\in \C^*$ be the $v_2$-coordinate of $p_{\pm 1}$. Let $x_0$ be a root of $U_{n-1}$. The open subset of $F_{x_0}$ visible in the $(y,v_2)$-coordinates is $\{(y,v_2):v_2=(x_0^2-1)/y^2\}\cong \Ast$. By the commutativity of the above diagram $\eta$ maps the fiber $F_{x_0}$ to $p_{T_n(x_0)}$, so the restriction of $\eta_2^*v_2$ to $F_{x_0}$ is constant, equal to $\lambda_{T_n(x_0)}$. We get $A_1(x_0)=0$ and $(x_0^2-1)b^2(x_0)=\lambda_{T_n(x_0)}^{-1}$. In particular, $A_1=U_{n-1}a$ for some $a\in \C[x]$. We obtain
\begin{equation}\label{eq:v_ver2}
\eta_2^*v_2=\frac{x^2-1}{(U_{n-1}ay+(x^2-1)b)^2}=\frac{v_2}{(U_{n-1}a+yv_2b)^2}.
\end{equation}
Putting $\bar v_2=v_2|_{E_1}$, we get $\eta_2^*\bar v_2/\bar v_2=(na(1))^{-2}$. Since $\eta_2$ fixes $p_1$, we obtain $a(1)=\pm 1/n$.
\end{proof}

\subsection{$\C^+$-equivariant GJC holds.}
 
\begin{prop}[$\C^+$-equivariant GJC]\label{lem:Cplus-equivariance} The $\C^+$-equivariant Generalized Jacobian Conjecture holds for $\Q$-homology planes with an effective $\C^+$action.
\end{prop}

\begin{proof} Let $S$ be a $\Q$-homology plane with an effective $\C^+$-action $\Lambda_t\:S\to S$, $t\in \C^+$. Suppose $\eta\in \Et_{\C^+}(S)$ is non-proper and put $n=\deg \eta >1$. The quotient morphism is an $\A^1$-fibration $\rho\:S\to \Spec (\C[x])$ respected by $\eta$. By Proposition \ref{prop:descent_for_A1-fibrations} $S$ is not a pseudo-plane, $\rho$ has exactly two degenerate fibers and there is a birational morphism $\sigma\:S\to \A^2$ such that $\rho=\pr_1\circ \sigma$ and an endomorphism $\eta_0\in \End(\A^2,\pr_1)$ given by 
\begin{equation*}
\eta_0(x,y)=(T_n(x),A(x)y+B(x)),
\end{equation*} where $A=U^2_{n-1}a$, $B=(x^2-1)U_{n-1}b$, and where $a,b\in \C[x]$ and $A(1)=\pm n$ such that $\eta_0\circ\sigma=\sigma\circ\eta$.

We now derive a contradiction from the assumption that $\eta$ is $\C^+$-equivariant.  First of all note that, in principle, a minimal $\C^+$-equivariant completion of $\rho$ can have a higher Picard rank than the minimal completion $\bar S$ used in the proof of Proposition \ref{prop:descent_for_A1-fibrations}. However, since $\Lambda_{t=0}=\id$, by continuity the $\C^+$-action on $\bar S$ extending $\Lambda$ maps fiber components to themselves, which by the minimality means that the boundary of such a completion contains no $(-1)$-curves in the fibers, hence it is also minimal among all smooth completions of $\rho$. Thus $\bar S$ is a $\C^+$-equivariant smooth completion of $S$. Then, since $1$-dimensional orbits of a $\C^+$-action are necessarily isomorphic to $\A^1$, we infer that the contractions $\sigma_i$ of $(-1)$-curves contained in fibers are $\C^+$-equivariant, so the diagram \eqref{eq:S_A^1-fibered_descent} is $\C^+$-equivariant. 

Suppose the action $\Lambda$ is not free over some $b\in \A^1$. Then $F_b:=\rho^{-1}(b)$ is contained in the fixed point set of $\Lambda$. Since $\eta$ is $\Lambda$-equivariant, the fixed point set contains $\bigcup_{k\geq 0}(\eta^{\circ k})^{-1}(F_b)$, where $\eta^{\circ k}$ denotes the $k$-th iteration of $\eta$. The latter cannot be dense in  $S$, because $\Lambda$ is effective, hence its image by $\rho$ is finite. So there exists $N\geq 1$ such that $Z=\bigcup_{k=0}^{N} (T_n^{\circ k})^{-1}(b)$ satisfies $T_n^{-1}(Z)\subseteq Z$. Since $T_n$ is finite of degree $\geq 2$, so is its restriction to $\A^1\setminus Z$. This is possible only if $e_{top}(\A^1\setminus Z)=0$, so $Z=\{b\}$ and $\deg T_n=n$ is the ramification index of $T_n$ at $b$. The latter gives $T_n^*\{b\}=n\cdot\{b\}$, and hence $n=1$ by Lemma \ref{lem:etale-basics}(1). Thus $\Lambda$ is a free $\C^+$-action.

Since $\sigma$ is an isomorphism over $x\neq \pm 1$, the $\C^+$-action $\Lambda^0$ induced by $\Lambda$ on $V_0$ has the form $\Lambda^0_t(x,y)=(x,y+t\cdot Q(x))$ for some polynomial $Q\in \C[x]$ of the form $Q(x)=c(x-1)^{m^-}(x+1)^{m^+}$ for some integers $m^-,m^+\geq 0$ and some $c\in \C^*$. We have $\eta_0\circ \Lambda^0_t=(T_n(x),A(x)(y+tQ(x))+B(x))$ and $\Lambda^0_t\circ \eta_0=(T_n(x),A(x)y+B(x)+tQ(T_n(x)))$. Since $\eta_0$ is $\C^+$-equivariant, we obtain $A(x)Q(x)=Q(T_n(x))$, i.e.\

\begin{equation}\label{eq:Cplus-a(x)}
A(x)=\left(\frac{T_n(x)-1}{x-1}\right)^{m^-}\left(\frac{T_n(x)+1}{x+1}\right)^{m^+}.
\end{equation}
In particular, $A(1)=(\frac{dT_n}{dx}(1))^{m^-}=n^{2m^-}$. But $A(1)=\pm n$, hence $n=1$; a contradiction. 
\end{proof}

\smallskip
\subsection{Miyanishi's counterexample expanded.}\label{ssec:Miyanishi}

Let again $S$ be a $\Q$-homology plane with an $\A^1$-fibration $\rho\:S\to \A^1$. If $\eta$ is a non-proper étale endomorphism of $S$ respecting $\rho$ (hence a counterexample to the Generalized Jacobian Conjecture for $S$) then by Proposition \ref{prop:descent_for_A1-fibrations} $\pi_1(S)\cong \Z_2*\Z_2$. A first example of this type has been constructed by Miyanishi \cite[2.4.3(2), 2.3.11]{Miyanishi-Lectures_on_polynomials}; it has degree $2$. We now make a digression to show how our computations from the previous section can be applied to construct counterexamples of any degree. Note that by Proposition \ref{lem:Cplus-equivariance} $\eta$ is not $\C^+$-equivariant for any effective $\C^+$-action on $S$.

\begin{example}[Miyanishi's $\Q$-acyclic counterexample with $\pi_1=\Z_2*\Z_2$]\label{ex:Miy_Z2stZ2} Let $(C,p_{-1},p_1,p_{\8})$ be a plane conic with a triple of distinct points on it. Denote by $F_{\pm 1}$ the lines tangent to $C$ at $p_{\pm 1}$ and by $F_\8$ the line joining $p_\8$ and the common point, say $q$, of $F_1$ and $F_{-1}$. Blow up once over $q$, three times over each of $p_{\pm 1}$, each time on the proper transform of $C$. Denote the last exceptional curves over $\pm 1$ by $E_{\pm 1}$. Let $\bar S$ be the resulting projective surface and $D$ be the total reduced transform of $F_1+F_\8+F_{-1}$ with $E_1+E_{-1}$ subtracted. Denote by $\bar \rho\:\bar S\to \P^1$ the $\P^1$-fibration induced by the linear system of the proper transform of $F_\8$ and by $\rho$ its restriction to $S=\bar S\setminus D$. The two degenerate fibers of $\rho$ are isomorphic to $2\A^1$ and have dual graph
$$\xymatrix@C=1.5pc@R=1.5pc{{-1}\ar@{-}[r] & {(-2)}\ar@{-}[r]\ar@{-}[d] & {-2}\ar@{.}[r] & {\bullet} \\ {} & {-2} & {} & { }}$$
where the black dot represents the exceptional curve over the point $F_1\cap F_\8 \cap F_{-1}$. The discriminant of the intersection matrix of $D$ is non-zero, so since $\rho(\bar S)=\#D$, the components of $D$ freely generate $\Pic(\bar S)\otimes_{\Z}\Q$. It follows that $S$ is $\Q$-acyclic.  It is also $\A^1$-fibered, so its Kodaira dimension is negative. 
\end{example}

In \cite[2.4.3(2), 2.3.11]{Miyanishi-Lectures_on_polynomials} Miyanishi argues geometrically that $S$ has an étale endomorphism of degree $2$. Namely, let  $(S', \rho', \varphi')$ be the normalized pullback of $\rho\:S\to\A^1$ and $\varphi\:\A^1\to \A^1$ given by $\varphi(x)=2x^2-1$. We have $\varphi^{-1}\{-1,1\}=\{-1,0,1\}$ and $\varphi$ is étale at $\pm 1$, so the degenerate fibers of $\rho'\:S'\to \A^1$ are $(\rho')^*(\pm 1)\cong 2\A^1$ and $(\rho')^*(0)\cong \A^1\sqcup \A^1$. If we remove from $S'$ one of the components of the latter fiber then one shows that we get a surface isomorphic with $S$. Let us note that the proof of the latter uses the fact that $S$ has a $\Z_2$ action lifting the $\Z_2$-action $x\mapsto -x$. 

We now construct étale endomorphism of every degree on Miyanishi's surface. 

\begin{prop}[Miyanishi's counterexample expanded]
Let $\rho\:S\to \A^1$ be the $\A^1$-fibration of Miyanishi's $\Q$-acyclic surface of Example \ref{ex:Miy_Z2stZ2} and let $\sigma:S\rightarrow \A^2$ be the birational morphism constructed in Proposition \ref{prop:descent_for_A1-fibrations}. Then for every $n\geq 2$ and every polynomial $b\in \C[x]$ such that 
 \begin{equation}\label{eq:Miy_condition_on_b_ver2}
 b(\cos(\frac{k\pi}{n}))=\pm \sqrt{-1}/\sin(\frac{k\pi}{n})\text{\ \ for\ \ } k=1,\ldots,n-1.
 \end{equation}  
the endomorphism $\eta_0\:\A^2\to \A^2$ defined by 
\begin{equation}\label{eq:Miy_eta_0}
\eta_0(x,y)=(T_n(x),\frac{1}{n}U^2_{n-1}(x)y+(x^2-1)U_{n-1}(x)b(x)),
\end{equation}
lifts via $\sigma:S\rightarrow \A^2$ to an étale endomorphism $\eta:S\rightarrow S$ of degree $n$ respecting $\rho$.

In particular, the monoid $\Et(S,\rho)$ contains endomorphisms of every positive degree.
\end{prop}

\begin{proof} 
Let $V_i$, $\sigma_i$ and $v_i$, $i=0,1,2,3$ be as in the proof of Proposition \ref{prop:descent_for_A1-fibrations}. We have $V_3=S$. Let again $F_x$ denote the fiber of $\rho_2=\rho_0\circ \sigma_0\circ \sigma_1\:V_2\to \A^1$ over $x$ and let $p_{\pm 1}\in F_{\pm 1}\cong \Ast$ denote the center of $\sigma_3$. We denote curves and their proper transforms by the same letters. In case of the Miyanishi surface we have $S=V_3$ and $p_{\pm 1}$ both have coordinates $(y,v_2)=(0,\lambda)$ (and $x=\pm 1$) for some $\lambda\in \C^*$, which we may assume to be equal to $1$. Put $\sigma=\sigma_0\circ\sigma_1\circ\sigma_2\circ\sigma_3\: S\to V_0$. 
Since the restriction of $\sigma$ to $S\setminus\rho^{-1}\{-1,1\}$ is an isomorphism, $\eta$ is well defined and quasi-finite on the complement of $\rho^{-1}(-1)\cup\rho^{-1}(1)\cup Z$, where $Z=\bigcup\{F_x:U_{n-1}(x)=0\}$.

The centers of the two blowups constituting $\sigma_0\:V_1\to V_0$ have coordinates $(x,y)=(\pm 1,0)$ and we compute that $\eta_0^{-1}\{(-1,0),(1,0)\}=Z\cup\{(-1,0), (1,0)\}$. By the universal property of a blowup $\eta_0$ lifts to a morphism $\eta_1\:V_1'\to V_1$, where $V_1'$ is the blowup of $V_0$ at $Z\cup \{(-1,0), (1,0)\}$. Since $Z$ is principal, we have $V_1'=V_1$. Moreover, the lift maps isomorphically the fibers over $x=\pm 1$ onto their images, because $\eta_0$ was a local analytic isomorphism at the centers of $\sigma_0$. Put $v_1=\frac{x^2-1}{y}$. The exceptional divisors $E_{1,\pm 1}$ over $x=\pm 1$ are described in the $(y,v_1)$-coordinates by $y=0$ (and $x=\pm 1$). By the definition of $\eta_0(x,y)$ and by \eqref{eq:Tn-Un_square_relation} we have 
$$\eta_1^*v_1=\frac{v_1 U_{n-1}(x)}{\frac{1}{n}U_{n-1}(x)+v_1b(x)}.$$

The centers $w_{\pm}$ of $\sigma_1$ are at the intersections of $E_{1,\pm1}$ with the proper transforms of the fibers $\{(\pm 1,y):y\in \C\}\subseteq V_0$ on $V_1$. The functions $y$ and $v_1=\frac{x^2-1}{y}$ are local parameters there. We check that $\eta_1^{-1}\{w_-,w_+\}=Z\cup \{w_-,w_+\}$. As before, $\eta_1$ lifts to an endomorphism $\eta_2$ of $V_2$ which maps isomorphically the fibers over $x=\pm 1$ onto their images. Put $v_2=\frac{v_1}{y}$.  The exceptional divisors of $\sigma_1$ over $x=\pm 1$ are $E_{2,\pm 1}=F_{\pm 1}$ and are described in the $(y, v_2)$-coordinates by $y=0$ (and $x=\pm 1$). We have 
$$\eta_2^*v_2=\frac{x^2-1}{(\frac{1}{n}U_{n-1}(x)y+(x^2-1)b(x))^2}.$$

Finally, the centers  $p_{\pm 1}\in E_{2,\pm 1}$ of $\sigma_2$ have $(y,v_2)$-coordinates $(0,1)$ (and $x=\pm 1$), and $y$ and $v_2-1$ are local parameters there. We compute that $\eta_2^{-1}\{p_{-1},p_1\}=Z\cup \{p_{-1},p_1\}$, hence $\eta_2$ lifts to an endomorphism $\eta$ of $S$ mapping the fibers over $x=\pm 1$ isomorphically onto their images (and respecting the induced $\A^1$-fibration on $S$). Thus $\eta$, which lifts $\eta_0$, has no base points. To show that $\eta$ is quasi-finite it remains to show that it does not contract curves in $Z$. Put $v_3=\frac{v_2-1}{y}=\frac{x^2-1-y^2}{y^3}$. We have  $$\eta^*v_3=\frac{(x^2-1)-(\frac{1}{n}yU_{n-1}(x)+(x^2-1)b(x))^2}{U_{n-1}(x)(\frac{1}{n}yU_{n-1}(x)+(x^2-1)b(x))^3}.$$ Since the zeros of $U_{n-1}$ are $\cos (k\pi/n)$, $k=1,\ldots, n-1$, by \eqref{eq:Miy_condition_on_b_ver2} there exists a polynomial $s\in \C[x]$ such that $$(x^2-1)b^2(x)=1-s(x)U_{n-1}(x).$$ In particular, $b(x_0)\neq 0$ if $U_{n-1}(x_0)=0$. We obtain
$$\eta^*v_3=\frac{(x^2-1)s(x)-\frac{1}{n^2}U_{n-1}(x)y^2-\frac{2}{n}y(x^2-1)b(x)}{(\frac{1}{n}yU_{n-1}(x)+(x^2-1)b(x))^3},$$ which for every root $x_0$ of $U_{n-1}$ gives $$\eta^*v_3|_{\rho^{-1}(x_0)}=\frac{s(x_0)-\frac{2}{n}yb(x_0)}{(x_0^2-1)^2b(x_0)^3}.$$ Since $b(x_0)\neq 0$, the fiber $\rho^{-1}(x_0)$ is not contracted. In fact, since the above expression is linear in $y$, which is a coordinate on $\rho^{-1}(x_0)$, this fiber is mapped isomorphically onto $E_{2,T_n(x_0)}$. Thus $\eta$ is quasi-finite, hence étale by Lemma \ref{lem:multiplicities}.  
\end{proof}

\subsection{Reduction to $\C^*$-actions on pseudo-planes $S(k,r,a)$.} 

We now complete the proof of Theorem \ref{thm:(G,S)=(Cst,S(kra))}. 

\begin{proof}[Proof of Theorem \ref{thm:(G,S)=(Cst,S(kra))}]
Let $\eta\in \Et_G(S)\setminus \Aut(S)$. By Proposition \ref{lem:Cplus-equivariance} $G$ does not contain a subgroup isomorphic to $\C^+$, so since it is infinite, it contains a subgroup isomorphic to $\C^*$, hence $S$ is a $\Q$-homology plane with an effective action of $\C^*$. 

Suppose the action is not hyperbolic. By Remark \ref{rem:non-hyperbolic_Cst_action} it is the action $\lambda\cdot(x,y)\to (\lambda^px, \lambda^qy)$ on $\A^2=\Spec (\C[x,y])$ for some coprime integers $p,q\geq 0$. We have $p,q\geq 1$, because otherwise $\eta$ respects the projection $\pr_x$ or $\pr_y$, contrary to Corollary \ref{cor:Et(S,p)-for_A1-fib_pseudo-planes}. In particular, the fixed point set of the action of $\C^*$ is $\{0\}$. Since $\eta$ is quasi-finite, it restricts to $U=\A^2\setminus\{0\}$. The morphism $\rho\:U\to \P^1$ given by $\rho(x,y)=[x^q:y^p]$ is a $\C^*$-equivariant $\Ast$-fibration respected by $\eta|_U$.  It has irreducible fibers, two of which are degenerate, namely $\rho^*([0:1])\cong q\Ast$ and $\rho^*([1:0])\cong p\Ast$. By Lemma \ref{lem:multiplicities} $[0:1]$ and $[1:0]$ are the only possible branching points of $\eta_\rho$. By the multiplicativity of the Euler characteristic $e_{top}(\eta_{\rho}^{-1}(\P^1\setminus \{[0:1],[1:0]\}))=0$, so if $\deg \eta_\rho>1$ then $\eta_\rho$ has exactly two branching points and two ramification points and they have the same ramification index. By Lemma \ref{lem:multiplicities} the latter index divides both $p$ and $q$, hence is equal to $1$. Thus $\deg\eta=\deg \eta_\rho=1$; a contradiction. 

We may therefore assume that the action of $\C^*$ on $S$ is hyperbolic, hence the quotient morphism $\rho\:S\to B=S/\C^*$ is an $\Ast$-fibration respected by $\eta$. Since $S$ is affine, $B$ is affine, so $B\cong\A^1$ by  \cite[Lemma 3.4.5.1(1)]{Miyan-OpenSurf}. By \cite[Lemma 1.3(5)]{MiyMa-hp_with_torus_actions} degenerate fibers of $\rho$ have reduced forms $\Ast$ and $\AuA$. By Lemma \ref{lem:factorization_sigma=id}(2) $\deg \eta=\deg \eta_\rho$, in which case Lemma \ref{lem:when_etaB_has_deg=1} says that $\rho$ has at least two non-reduced fibers. Then, since $\kappa(S)=-\8$, we infer from \cite[Theorem 3.4.6.2]{Miyan-OpenSurf} that $\rho$ has precisely two degenerate fibers and they are isomorphic to $k\Ast$ and $\A^1\cup_{\{0\}}r\A^1$ for some integers $k,r\geq 2$. Finally, Lemma \ref{lem:Cst-pseudo-planes} says that there exists $a\in \{1,2,\ldots,k-1\}$ coprime with $k$ such that $S$ is $\C^*$-equivariantly isomorphic to $S(k,r,a)$ with the $\C^*$-action induced from \eqref{eq:Cst-action_Sti(k,r)}. By the discussion in Example \ref{ex:S(k,r,a)}, $S$ is an affine pseudo-plane.

It remains to show that $G\cong \C^*$. More precisely, the action of $\Z_k$ on the universal cover $\ti S\cong \ti S(k,r)$  commutes with the $\C^*$-action \eqref{eq:Cst-action_Sti(k,r)} and maps orbits isomorphically onto orbits, so the induced homomorphism $\Aut(\ti S)\to \Aut(S)$ maps the subgroup $\C^*$ isomorphically onto $\C^*$. We will show that $G$ is exactly the latter $\C^*$. Let $\sigma\in G$. By Lemma \ref{lem:lifting_eta} $\eta$ lifts to a non-proper étale endomorphism $\ti \eta$ of $\ti S$ and $\sigma \in G$ lifts to an automorphism $\ti \sigma$ of $\wt S$. By Theorem \ref{thm:GJC_fails_v2} in the next section, there exist $\lambda\in \C^*$ and a $k$-th root of unity $\eps$, such that
$$\ti \eta(x,y,z)=(\eps\lambda xz^{1-\alpha}R_2(1-z^k), \eps^{-r}\lambda^{-r}yR_0(1-z^k),\eps^{-a}z^{\alpha}R_1(1-z^k))$$ for some $\alpha\in\{0,1\}$ and some polynomials $R_i(t)$. By \cite[Theorem 4.4]{MiyMa-hp_with_torus_actions} we have
\begin{equation}
\ti \sigma(x,y,z)=(tx,t^{-r}y+f(x)F(x,z),\zeta z+x^rf(x)),
\end{equation}
where $t\in \C^*$, $\zeta^k=1$ and $F(x,y)$ is some polynomial in $x,y$ uniquely determined by $f(x)$ and the equation of $\ti S$. Put $\bar z=\zeta z+x^rf(x)$. Since $\eta$ and $\sigma$ commute, their lifts commute up to the action of $\Z_k$ given by \eqref{eq:Zk-action}, i.e.\ $\ti \eta\circ\ti \sigma=\mu*_a(\ti \sigma\circ \ti \eta)$ for some $\mu\in \Z_k$. Composing with the projection onto $x$ and dividing by $\eps\lambda t x$ we get $$\bar z^{1-\alpha}R_2(1-\bar z^k)=\mu z^{1-\alpha}R_2(1-z^k)\in \C[z, x,x^{-1}].$$ If $\alpha=0$ then substituting $z=0$ we get $\bar zR_2(1-\bar z^k)=0$, so since $\deg R_2\geq 0$, we obtain $\bar z=x^rf(x)$, so $f=0$ in this case. If $\alpha=1$ then substituting $z=0$ leads to the equality $R_2(1-(x^rf(x))^k)=\mu R_2(1)$. But since $\deg \eta>1$, Lemma \ref{lem:Ri-condition} implies that $\deg R_2\geq \alpha$, so in the latter case we again get $f=0$. Thus $f=0$ and hence $\ti \sigma(x,y,z)=(tx, t^{-r}y,\zeta z)$. Choosing $\nu$ such that $\nu^a=\zeta$ we get a new lift of $\sigma$ given by $\nu*_a\ti \sigma(x,y,z)=((\nu t)x,(\nu t)^{-r}y,z)$. Thus $\sigma$ has a lift given by \eqref{eq:Cst-action_Sti(k,r)} for $\lambda=\nu t\in \C^*$. It follows that $\sigma$ is identical with the induced action by $\lambda=\nu t$ on $S$.
\end{proof}

\section{Proof of Theorem \ref{thm:GJC_fails}. Non-proper $\C^*$-equivariant étale endomorphisms}\label{sec:Et_Cst}

Let $S$ be a $\Q$-homology plane of negative Kodaira dimension. By Theorem \ref{thm:(G,S)=(Cst,S(kra))} when looking for  counterexamples to the equivariant Generalized Jacobian Conjecture for infinite algebraic groups we are reduced to the case of pseudo-planes $S(k,r,a)$ for $r,k\geq 2$ (see Example \ref{ex:S(k,r,a)}) endowed with their unique hyperbolic $\C^*$-action. In this section we prove Theorem \ref{thm:GJC_fails}, which gives a full description of counterexamples to the $\C^*$-equivariant Generalized Jacobian Conjecture for these surfaces.

\subsection{Necessary conditions on $\eta_\rho$.}\label{ssec:Shabat_poly}

As discussed above, to prove Theorem \ref{thm:GJC_fails} we may assume $S=S(k,r,a)$ for some $r,k\geq 2$ and $a\in \{1,2,\ldots,k-1\}$ coprime with $k$. We denote the hyperbolic $\C^*$-action on $S$ by $\Lambda$ and the quotient $\Ast$-fibration by $$\rho\:S\to B=\Spec (\C[t]).$$ If $\eta$ is a $\C^*$-equivariant endomorphism of $S$ then $\eta\in \Et(S,\rho)$, so we have a commutative diagram \[\xymatrix{ S \ar[r]^{\eta} \ar[d]_{\rho} & S \ar[d]^{\rho} \\ B \ar[r]^{\eta_\rho} & B }\] 
where $\eta_\rho$ is some endomorphism of $B=\Spec (\C[t])$. By Lemma \ref{lem:factorization_sigma=id}(2) $\deg \eta=\deg \eta_\rho$. Therefore, to construct counterexamples to the Generalized Jacobian Conjecture we first need to understand the conditions met by the induced endomorphisms $\eta_\rho$ of degree bigger than $1$. As we will see, they are (specific) Belyi-Shabat polynomials, but not necessarily the Chebyshev polynomials (cf.\ Section \ref{rem:Shabat_vs_Gal}).

\medskip
Let $k,r\geq 2$ and $S=S(k,r,a)$ be as above. Assume $\eta\in \Et(S,\rho)$ is non-proper.  
By Example \ref{ex:S(k,r,a)} $\rho$ has exactly two degenerate fibers and there is a unique choice of a coordinate $t$ on $B$ such that the fibers are
\begin{equation}
F_1=\rho^*(1)=kA_1\cong k \Ast {\ \ and\ \ } F_0=\rho^*(0)=A_0\cup rA_2\cong \A^1\cup_{\{0\}}r\A^1,
\end{equation}
where $k$ is the order of the cyclic group $\pi_1(S)$. 

\bigskip
\begin{lem}[Conditions on $\eta_\rho$]\label{lem:Ri-condition} Let $S=S(k,r,a)$, $k,r\geq 2$ and $\rho\:S\to \Spec (\C[t])$ be as above. Assume $\eta\in \Et(S,\rho)$ and put $\alpha=\eta_\rho(1)\in \{0,1\}$. Then $\eta_\rho$ has degree $d:=\deg \eta$, has at most two critical values, $0$ and $1$, and
\begin{equation}\label{eq:alpha,k,r-condition}
\text{either\ } \alpha=1 \text{\ \ or\ \ } k|r \text{\ and\ } \alpha=0.
\end{equation} Moreover, there exist polynomials $R_0, R_1, R_2$ of degrees $d_0, d_1, d_2$ respectively such that 
\begin{equation}\label{eq:Ri-condition}
\eta_\rho(t)=t(1-t)^{(1-\alpha)\frac{r}{k}}R_0(t)R_2^r(t)=1-(1-t)^{\alpha}R_1^k(t),
\end{equation}

\begin{equation}\label{eq:Ri-degrees}
d_2=\frac{d -\alpha-r(1-\alpha)}{k(r-1)},\  d_1=d_2(r-1)+(1-\alpha)\frac{r}{k},\ d_0=(d_2k+1-\alpha)(r-1-\frac{r}{k})
\end{equation}
and 
\begin{equation}\label{eq:Ri-constant_terms}
(1-t)R_0R_1R_2\text{\ is separable,\ }R_1(0)=R_2(0)=1,\ R_0(0)\neq 0.
\end{equation}
\end{lem}

\begin{proof} By Lemma \ref{lem:factorization_sigma=id}(2) $\eta$ maps general fibers of $\rho$ isomorphically to their images, hence $\deg \eta=\deg \eta_\rho$. By Lemma \ref{lem:multiplicities} if $C$ is a component of a fiber of $\rho$ and $\rho(C)=t_0\in \C$ then the multiplicity of $t_0$ as a root of $\eta_\rho(t)-\eta_\rho(t_0)$ is $\mu(\eta(C))/\mu(C)$. In particular, it equals $1$ if  $\eta_\rho(t_0)\neq 0,1$ (then $\mu(\eta(C))=1$) or if $t_0$ is a fixed point of $\eta_\rho$. The former implies that the only critical values of $\eta_\rho$ are $0, 1$. The latter holds for instance for $t_0=0$, because $F_0$ is the only reducible fiber of $\rho$, and hence $\eta(F_0)\subset F_0$. Since $F_1$ is a multiple fiber, we have $\eta(F_1)\subset F_0\cup F_1$. Let $\alpha\in \{0,1\}$ be the multiplicity of $1$ in $\eta_\rho(t)-\eta_\rho(1)$. With the above choice of $t$ we have $\alpha=\eta_\rho(1)$.

Let $\{\alpha_1,\ldots,\alpha_{d_0}\}\subset \C$ and $\{\beta_1,\ldots,\beta_{d_2}\}\subset \C$ be all distinct points such that the fibers over them are smooth and are mapped by $\eta$ to $A_0$ and $A_2$ respectively. They appear in $\rho^*(0)$ with multiplicities $1$ and $r$ respectively, so we can write $\eta_\rho$ as $\eta_\rho(t)=t(1-t)^{(1-\alpha)m} R_0(t)R_2^r(t)$, where $m$ is a positive integer, $t(1-t) R_0(t)R_2(t)$ separable, say $R_0(t)=C_0\prod_{i=1}^{d_0}(\alpha_i-t)$, $C_0\neq 0$ and $R_2(t)=(\prod_{i=1}^{d_2}\beta_i)^{-1}\cdot \prod_{i=1}^{d_2}(\beta_i-t)$. If $\alpha=1$ then we may put $m=r/k$, and if $\alpha=0$ then $\eta_\rho(1)=0$, so $\eta(A_1)\subseteq A_2$ and $m=\mu(A_2)/\mu(A_1)=r/k$.

We analyze $\eta_\rho^*(1)$ in a similar way. Let $\{\gamma_1,\ldots,\gamma_{d_1}\}\subset \C$ be all distinct points such that the fibers over them are smooth and mapped by $\eta$ to $A_1$. By Lemma \ref{lem:multiplicities} they appear in $\eta_\rho^*(1)$ with multiplicity $k$, hence $1-\eta_\rho(t)=(1-t)^{\alpha}R_1^k(t)$, where $R_1(t)=C_1(\prod_{i=1}^{d_1}\gamma_i)^{-1}\cdot \prod_{i=1}^{d_1}(\gamma_i-t)$, $C_1\neq 0$ and $R_1(1)\neq 0$. Since $\eta_\rho(0)=0$, we have $C_1^k=1$, so we may assume $C_1=1$.

It remains to prove \eqref{eq:Ri-degrees}. By \eqref{eq:Ri-condition} $$d=1+d_0+r d_2+(1-\alpha)\frac{r}{k}=\alpha+k d_1.$$ Since $\eta_\rho$ has no branching points other then $0,1$, the Riemann-Hurwitz formula gives
$$d-1=d \cdot e_{top}(\A^1)-e_{top}(\A^1)=d_1(k-1)+d_2(r-1)+(1-\alpha)(\frac{r}{k}-1).$$ 
We get $\alpha+kd_1-1=d-1=d_1(k-1)+d_2(r-1)+(1-\alpha)(\frac{r}{k}-1)$, hence $d_1=d_2(r-1)+(1-\alpha)\frac{r}{k}$. Since $d_1=(d-\alpha)/k$, we get the first two formulas in \eqref{eq:Ri-degrees}.
We obtain also $d_0=kd_1-rd_2+(1-\alpha)(-1-\frac{r}{k})=d_2(k(r-1)-r)+(1-\alpha)(r-1-\frac{r}{k})=(d_2k+1-\alpha)(r-1-\frac{r}{k}).$
\end{proof}

\subsection{The universal covers and formulas for étale endomorphisms.}

We now prove a stronger and more detailed version of Theorem \ref{thm:GJC_fails}.

\begin{thmA}\label{thm:GJC_fails_v2} Let $\eta$ be an étale endomorphism of the pseudo-plane $S(k,r,a)$, $k,r\geq 2$ which respects the quotient $\Ast$-fibration. Then there exists a lift of $\eta$ to $\ti S(k,r)$ given by 
\begin{equation}\label{eq:eta_tilda_formula}
\ti \eta(x,y,z)=(\lambda xz^{1-\alpha}R_2(1-z^k),\lambda^{-r}yR_0(1-z^k),z^{\alpha}R_1(1-z^k)),
\end{equation}
for some $\lambda\in\C^*$, some $\alpha\in\{0,1\}$ satisfying \eqref{eq:alpha,k,r-condition} and such that $a=1$ if $\alpha=0$, and some polynomials $R_0, R_1, R_2\in \C[t]$ satisfying \eqref{eq:Ri-condition}-\eqref{eq:Ri-constant_terms}. In particular, $\ti \eta$ and $\eta$ are $\C^*$-equivariant, $d:=\deg \eta=\deg \ti \eta$ and we have
\begin{equation}\label{eq:deg_modulo_k(r-1)}
d\equiv \alpha+r(1-\alpha)\mod k(r-1),
\end{equation}
hence $d\equiv \alpha \mod k$.

Conversely, given integers $k,r\geq 2$, $a\in \{1,\ldots,k-1\}$ coprime with $k$, and $\alpha\in\{0,1\}$ satisfying \eqref{eq:alpha,k,r-condition}, for every positive integer $d$ satisfying \eqref{eq:deg_modulo_k(r-1)} there exist polynomials $R_0,R_1,R_2$ satisfying \eqref{eq:Ri-condition}-\eqref{eq:Ri-constant_terms} and then the formula \eqref{eq:eta_tilda_formula} defines a $\C^*$-equivariant étale endomorphism of $\ti S(k,r)$ of degree $d$ which descends to a $\C^*$-equivariant étale endomorphism of degree $d$ of $S(k,r,a)$.
\end{thmA}

Almost all statements of Theorem \ref{thm:GJC_fails} follow immediately from the above result, except for the claim concerning finite covers. For the latter note that in case $\alpha=1$ the lift $\wt \eta$ is $\Z_k$-equivariant with respect to \eqref{eq:Zk-action} for all $a$ and in case $\alpha=0$ it is $\Z_k$-invariant, because $a=1$ and $k|r$.

\begin{proof}[Proof of Theorem \ref{thm:GJC_fails_v2}] Fix $k,r\geq 2$ and $a\in \{1,\ldots,k-1\}$ coprime with $k$. Put $\ti S=\ti S(k,r)$ and let $\pi\:\ti S\to S=S(k,r,a):=\ti S/\Z_k$ be the quotient morphism of the $\Z_k$-action given in \eqref{eq:Zk-action}, where $\eps$ is a fixed primitive $k$-th root of unity; it is the universal covering morphism for $S$. The $\Ast$-fibration, which is the quotient morphism of the $\C^*$-action, is $\ti \rho=\pr_z|_{\ti S}\:\ti S\to \Spec (\C[z])$. Its degenerate fibers are $$\wt \rho^*(\eps^j)=A_0^j\cup rA_2^j\cong\A^1\cup_{\{0\}}r\A^1,$$ where $A_2^j$ and $A_0^j$ zeros of the ideals $(x,z-\eps^j)$ and $(y,z-\eps^j)$ respectively. We denote by $\rho\:S\to\Spec (\C[t])$ the induced $\Ast$-fibration on $S$ and by $\pi'\: \Spec \C[z]\to \Spec \C[t]$ the morphism of the quotients by $\C^*$-actions induced by $\pi$. Since the induced $\Z_k$-action on $\Spec (\C[z])$ is $z\mapsto \eps^{-a}z$, we see that $z^k$ is a coordinate on $\Spec (\C[t])$. Recall that $\rho^*(0)$ is the only reducible fiber and that $\rho^*(1)\cong k\Ast$. Since the fixed fiber of the $\Z_k$-action is over $z=0$ and since the reducible ones are over roots of unity, we get $t=\pi'(z)=1-z^k$.

Assume $\eta\in \Et(S,\rho)$. Since $\ti S$ is simply connected, by Lemma \ref{lem:lifting_eta} $\eta$ has a lift $\ti \eta\in \Et(\ti S,\ti \rho)$. Let $R_i$, $i=1,2,3$ be as in Lemma \ref{lem:Ri-condition}. By the latter lemma we only need to prove the formula for $\ti \eta$ and that $a=1$ in case $\alpha=0$. Write $\ti \eta(x,y,z)=(\eta_1(x,y,z),\eta_2(x,y,z),\eta_3(z))$, where $(x,y,z)\in \ti S\subset \C^3=\Spec (\C[x,y,z])$ and $\eta_i$ are regular functions on $\ti S$. The function $\eta_3$ depends only on $z$, because $\eta_3=\wt \eta_{\wt \rho}$, i.e.\ $\eta_3$ is the induced morphism on the quotient of the $\C^*$-action. We have $$\eta_\rho\circ\pi'\circ \ti \rho=\eta_\rho\circ \rho\circ \pi=\rho\circ \eta\circ \pi=\rho\circ \pi\circ\ti \eta=\pi'\circ\ti\rho\circ \ti\eta=\pi'\circ\ti \eta_{\ti \rho}\circ \ti\rho,$$ hence $\eta_\rho\circ\pi'=\pi'\circ\ti \eta_{\ti \rho}$, because $\ti \rho$ is surjective. We get $\eta_\rho(1-z^k)=1-\eta_3^k(z)$, so by Lemma \ref{lem:Ri-condition} $\eta_3^k(z)=(z^{\alpha}R_1(1-z^k))^k$, hence composing $\ti \eta$ with an action of some $\eps\in \Z_k$ if necessary, we may assume $\eta_3(z)=z^{\alpha}R_1(1-z^k)$. Since $\ti \eta (x,y,z)\in \ti S$, we have 
$$(\eta_1^r\eta_2)(x,y,z)=\eta_3^k(z)-1=-\eta_\rho(1-z^k)=(z^k-1)z^{(1-\alpha)r}R_0(1-z^k)R_2^r(1-z^k).
$$ But $z^k-1=x^ry$, so for every $(x,y,z)\in \ti S$ we have \begin{equation}\label{eq:eta_i_relation}\eta_1^r(x,y,z)\eta_2(x,y,z)=(xz^{1-\alpha}R_2(1-z^k))^r(yR_0(1-z^k)).\end{equation}

By \eqref{eq:Ri-condition} the curve $A_2$, the smooth fibers of $\rho$ over zeros of $R_2(t)$ and $A_1$ in case $\alpha=0$ are mapped by $\eta$ to $A_0\cup rA_2$. Since the ramification indices of $\eta_\rho$ at the corresponding points of $\Spec (\C[t])$ are respectively $1$, $r$ and $r/k$, by Lemma \ref{lem:multiplicities} images of these curves by $\eta$ have all multiplicity $r$, that is, they are all mapped to $A_2$. Therefore, their inverse images by $\pi$, which are described by the ideals $(x,1-z^k)=(x)$, $(R_2(1-z^k))$ and $z^{1-\alpha}$ respectively, are mapped by $\wt \eta$ to $\pi^{-1}(A_2)$. This implies that $x\in \C[\ti S]$ vanishes on their images by $\eta$ and hence $\eta_1\in \C[\ti S]$ vanishes on the zero set of $xz^{1-\alpha}R_2(1-z^k)$. Similarly, since $\pi^{-1}(A_0)$ is described by the ideal $(y,1-z^k)=(y)\subseteq \C[\ti S]$, we get that $\eta_2\in \C[\ti S]$ vanishes on the zero set of $yR_0(1-z^k)$. By \eqref{eq:Ri-constant_terms} we infer that $xz^{1-\alpha}R_2(1-z^k)|\eta_1$ and $yR_0(1-z^k)|\eta_2$, hence by \eqref{eq:eta_i_relation} $\eta_1(x,y,z)=\lambda xz^{1-\alpha}R_2(1-z^k)$ and $\eta_2(x,y,z)=\lambda^{-r}yR_0(1-z^k)$ for some invertible $\lambda \in \C[\ti S]$. Since $\ti S$ contains $\A^2$, $\lambda$ is a constant, which gives \eqref{eq:eta_tilda_formula}.

Assume $\alpha=0$. Then $k|r$ so we have $\eps*_a(x,y,z)=(\eps x,y,\eps^{-a}z)$, hence $$\ti \eta(\eps*_a(x,y,z))=\ti \eta(\eps x,y,\eps^{-a}z)=(\lambda\eps^{1-a}xzR_2(1-z^k),\lambda^{-r}yR_0(1-z^k),R_1(1-z^k)).$$ Pick $(x,y,z)\in \ti S$ such that $xzR_2(1-z^k)\neq 0$. Since $\ti \eta$ descends to $\eta$, the points $\ti \eta(\eps*_a(x,y,z))$ and $\ti \eta(x,y,z)=(\lambda xzR_2(1-z^k),\lambda^{-r}yR_0(1-z^k),R_1(1-z^k))$ are in the same $\Z_k$-orbit, hence $\eps^{1-a}=1$, i.e, $a=1$. Note that for $\alpha=1$ we have $\ti \eta(\eps*_a(x,y,z))=\eps*_a\ti \eta(x,y,z)$, so there is no additional condition on $a$.

To prove the inverse implication assume that $\alpha\in \{0,1\}$, $k,r\geq 2$ are integers satisfying \eqref{eq:alpha,k,r-condition} and $d>0$ is an integer satisfying \eqref{eq:deg_modulo_k(r-1)}. Define (non-negative) integers $d_0,d_1, d_2$ by \eqref{eq:Ri-degrees}. Put $n=2$, $\lambda_1=((r)_{d_2},((1-\alpha)\frac{r}{k})_{1-\alpha},(1)_{d_1},1)$  and $\lambda_2=((k)_{d_1},(\alpha)_\alpha)$, where $(a)_b$ denotes the sequence $a,a\ldots,a$ of length $b$. By Thom's Lemma (Lemma \ref{lem:Thom}) there exists a polynomial $\varphi\:\A^1\to \A^1$ with ramification profile $(\lambda_1,\lambda_2)$ and $(0,1)$ as the underlying branching locus. Let $x_1$ be a point of ramification index $1$ mapping to $0$. If $\alpha=1$ then let $x_2$ be a point of ramification index $1$ mapping to $1$, otherwise let it be a point of ramification index $\frac{r}{k}$ mapping to $0$. Changing $\varphi(t)$ to $\varphi(t)=\varphi((x_2-x_1)t+x_1)$ we may assume $x_1=0$ and $x_2=1$, so $\varphi$ can be written as in \eqref{eq:Ri-condition} with $t(1-t)R_0R_1R_2$ separable. Multiplying $R_1, R_2$ by appropriate constants we may assume $R_1(0)=R_2(0)=1$. Now the formulas \eqref{eq:eta_tilda_formula} define a $\C^*$-equivariant endomorphism $\ti \eta$ of $\ti S$. Since $a=1$ in case $\alpha=0$, by the above argument it descends to a $\C^*$-equivariant endomorphism of the quotient by the $\Z_k$-action. By Lemma \ref{lem:multiplicities} $\ti \eta$ is étale, so $\eta\circ \pi=\pi\circ\ti\eta$ is étale. Since $\pi\:\ti S\to S$ is surjective and étale, $\ti \eta$ is étale.
\end{proof}

\begin{rem}\label{rem:first}\ \begin{enumerate}[(1)]
\item Our counterexamples to the Generalized Jacobian Conjecture seem to be the first ones in literature which are simply connected and have negative logarithmic Kodaira dimension, and also the first ones which are simply connected and rational. Note that in \cite{Miyanishi-Lectures_on_polynomials} the Remark on page 80 and the comment above Example 2.2.10 are incorrect. Indeed, in the first case the given endomorphism of $\{x^rz+y^d=1\}$, namely $\eta(x,y,z)=(x,y^n,z(1-y^{nd})/(1-y^d))$, is not étale for $y=0$. In the second case $\widetilde{X}$ is not simply connected, as its fundamental group is $\Z_2*\Z_2$.
\item It follows from Example 2.2.2 and Theorem 2.2.8 in \cite{Miyanishi-Lectures_on_polynomials} that the Generalized Jacobian Conjecture fails for the triple cover of the complement of a smooth planar cubic. This is a simply connected surface with logarithmic Kodaira dimension equal to $0$. This example leads also to a $\Q$-acyclic counterexample of logarithmic Kodaira dimension $0$ (see Example 2.2.3(5) loc.\ cit.).
\end{enumerate}
\end{rem}

\subsection{Examples.}
The proof of existence of $\eta\in \Etc(S(k,r,a))$ of degree $d$ in Theorem \ref{thm:GJC_fails_v2} is based on the existence of a specific Belyi-Shabat polynomial as in Lemma \ref{lem:Ri-condition}, which follows from Thom's Lemma \ref{lem:Thom}. In practice, finding the appropriate polynomials often leads to complicated algebraic equations. We analyze the following relatively simple cases.

\begin{example}[Explicit formulas for $\ti S(2,2)$]
Let $\ti \eta$ be a $\C^*$-equivariant endomorphism of $$\ti S=\ti S(2,2)=\{x^2y=z^2-1\}.$$ We claim that $$\ti \eta(x,y,z)=(x\lambda^{-1}U_{d-1}(z),\lambda^2y,T_d(z)),$$ where $\lambda\in \C^*$ and $T_d$, $U_d$ denote the Chebyshev polynomials of degree $d$ of the first and second kind respectively.

To see this note first that by \eqref{eq:Ri-degrees} $d_0=0$, $d_1=(d-\alpha)/2$, $d_2=(d+\alpha)/2-1$. Choose $\lambda=\sqrt{R_0(0)}\neq 0$, $P(z):=z^\alpha R_1(1-z^2)$ and $Q(z):=\lambda z^{1-\alpha}R_2(1-z^2)$. Then $P(1)=R_1(0)=1$,  $\deg P=d$ and $\deg Q=d-1$. With $t=1-z^2$ the relation \eqref{eq:Ri-condition} reads as 
\begin{equation}\label{eq:Pn_Qn:1}
P^2-1=(z^2-1)Q^2.
\end{equation} Differentiating we get $PP'=Q(zQ+(z^2-1)Q')$, so $Q|PP'$. But the above equation implies that  $P$ and $Q$ have no common roots, so $Q|P'$ and hence $P'=\beta Q$ for some $\beta\in \C^*$. We get 
\begin{equation}\label{eq:Pn_Qn:2}
\beta P=zQ+(z^2-1)Q'.
\end{equation}
Now \eqref{eq:Pn_Qn:1} implies that the critical points of $P$ are non-degenerate and the critical values are $\pm 1$. Then by \eqref{eq:Pn_Qn:2} $\pm 1$ are not critical points of $P$. By Lemma \ref{lem:Chebyshev_characterization} $P=T_d$, where $T_d$ is a Chebyshev polynomial of the first kind of degree $d$. By \eqref{eq:Tn-Un_square_relation} we have $Q=U_{d-1}$, where $U_{d-1}$ is a Chebyshev polynomial of the second kind. This gives the above formula for $\ti \eta$.
\end{example}

\begin{example}[Cyclic Galois cases]\label{ex:cyclic_Galois} Assume $S(k,r,a)$, $k,r\geq 2$ and $\eta\in \Etc(S(k,r,a))$ are such that $\eta_\rho(t)$ is cyclic Galois. Then $\frac{d}{dt}\eta_{\rho}=0$ has exactly one (multiple) solution. By \eqref{eq:Ri-condition} $R_1|\frac{d}{dt}\eta_{\rho}$, so since $R_1$ is separable, we get $d_1=1$, hence $R_1(t)=\beta t+1$ for some $\beta\in \C^*$. By \eqref{eq:Ri-condition} $\deg \eta=\deg \eta_\rho=k+\alpha\in\{k,k+1\}$. For $\alpha=1$ we get $$\frac{d}{dt}\eta_{\rho}(t)=R_1^{k-1}(t)(R_1(t)+k(t-1)\frac{d}{dt}R_1(t))=(\beta t+1)^{k-1}((\beta t+1)(k+1)-k(\beta+1)),$$ which has exactly two distinct roots, because $k\geq 2$ and $\beta+1=R_1(1)\neq 0$. Thus $\alpha=0$ and hence $a=1$, so $\deg \eta=k|r$. By \eqref{eq:Ri-degrees} $d_2=0$ and hence $d_0=k-2$. Thus the only possibility is $S_k=S(k,k,1)\cong \{u(1+uv)=w^k\}$, see \eqref{eq:S(k,rk,1)}, and $\eta=\pi\circ j$, with $j$ given in Corollary \ref{cor:embedding_S_into_wtS}. Then $\eta_\rho(t)=1-((\eps-1)t+1)^k$.
\end{example}
 
\begin{example}[A non-Galois case] Assume $(k,r)=(3,2)$. By Lemma \ref{lem:Ri-condition}  $\alpha=1$ and hence $d_0=d_1=d_2=\frac{1}{3}(d-1)$, where $d=\deg \eta_\rho$. By Example \ref{ex:cyclic_Galois} in this case $\eta_\rho$ is not cyclic Galois. It is given by the formula 
\begin{equation}
\eta_\rho(t)=tR_0(t)R_2^2(t)=1-(1-t)R_1^3(t).
\end{equation} 
Write $R_1(t)=a_{d_0}t^{d_0}+\ldots+a_1t+1$, $a_{d_0}\neq 0$. To find $a_i$, $i=1,\ldots, d_0$ note that if $t_0$ is a root of $R_2$ then it is a multiple root of $\eta_\rho$, so $\eta_\rho(t_0)=\frac{d}{dt}\eta_\rho(t_0)=0$, hence $R_1(t)+3(t-1)\frac{d}{dt}R_1(t)$ vanishes at $t_0$. Since $R_2$ is separable and $d_1=d_2$, we get that the polynomials
$R_2(t)$ and $R_1(t)+3(t-1)\frac{d}{dt}R_1(t)$ are equal up to a multiplication by some non-zero constant. Then $R_0$ as above exists if and only if 
\begin{equation}\label{eq:(k,r)=(3,2)_condition}
(R_1(t)+3(t-1)\frac{d}{dt}R_1(t))^2\mid 1-(1-t)R_1^3(t),
\end{equation}
which gives equations on $a_i$, $i=1,\ldots, d_0$.

Consider the case $d_0=1$. Then $d_1=d_2=1$, $d=4$ and $R_1(t)=a_1t+1$ for some $a_1\neq 0$. The condition \eqref{eq:(k,r)=(3,2)_condition} gives $a_1=\frac{1}{3}(-7+\imath \sqrt{2})$, where $\imath$ is a square root of $-1$. Then $R_0(t)=6(1+2\imath \sqrt{2})t+8-4\imath \sqrt{2}$.

Consider the case $d_0=2$. Then $d_1=d_2=2$ and $d=7$. We get $R_2(t)=(-3a_1+1)^{-1}(7a_2t^2-(6a_2-4a_1)t+(-3a_1+1)).$ Solving the condition \eqref{eq:(k,r)=(3,2)_condition} one gets six solutions, one of the simplest-looking are $(a_1,a_2)=(\frac{1}{24}(87+91\imath \sqrt{7}),-\frac{1}{24}(139+63 \imath\sqrt{7}))$, where $\imath$ is a square root of $-1$. In this case $R_0(t)=\frac{21}{128}((112 - 48\imath\sqrt{7}) + (644+ 268\imath\sqrt{7}) t - (999+85\imath\sqrt{7})t^2)$.
\end{example}

\section{Proof of Theorem \ref{thm:deformations}. The hypersurfaces $S(k,\bar r k,1)$ and deformations.}\label{sec:S(k,kr,1)}

Let $S(k,r,a)$, $k,r\geq 2$ (see Example \ref{ex:S(k,r,a)}) and let $\eta\in \Etc(S(k,r,a))$. By Theorem \ref{thm:GJC_fails_v2} some lift of $\eta$ to the universal cover $\ti S(k,r)$ is given by the formula \eqref{eq:eta_tilda_formula}. The case $\alpha=0$ is special, because then $k\mid r$ and $a=1$. We analyze it in detail and then in Section \ref{ssec:deformations} we use it to prove Theorem \ref{thm:deformations}.

\subsection{Surfaces $S(k,\bar rk,1)$.}\label{ssec:S(k,kr,1)}
Assume $r=\bar rk$ for some $\bar r\geq 1$ and $a=1$ (we do not assume $\alpha=0$).  By definition $S(k,\bar rk,1)$ is the quotient of $\ti S(k,\bar r k)=\{(x^k)^{\bar r}y=z^k-1\}$ by the $\Z_k$-action $\eps*(x,y,z)=(\eps x,y,\eps^{-1}z)$. We compute that the homomorphism $\C[u,v,w]\to \C[\ti S(k,\bar r k)]^{\Z_k}$ given by $f(u,v,w)\mapsto f(x^k, y, xz)$ is surjective with the kernel generated by $u(1+u^{\bar r}v)-w^k$, hence 
\begin{equation}\label{eq:S(k,rk,1)}
S(k,\bar rk,1)\cong \{(u,v,w):u(1+u^{\bar r}v)=w^k\}\subseteq \Spec (\C[u,v,w]),
\end{equation}
where the quotient morphism $\pi\:\tilde S(k,\bar r k)\to S(k,\bar r k,1)$ is given by
\begin{equation}
\pi(x,y,z)=(x^k,y,xz).
\end{equation}
The $\C^*$-action is now 
\begin{equation}\label{eq:Cst-action_on_S(k,kr,1)}
\lambda\cdot (u,v,w)=(\lambda^k u,\lambda^{-\bar r k}v,\lambda w)
\end{equation} 
and the quotient $\Ast$-fibration $\rho\:S(k,\bar rk,1)\to \Spec (\C[t])$ is $\rho(u,v,w)=-u^{\bar r} v$. It has two degenerate fibers:
$$F_1=\rho^*(1)\cong k\Ast \text{\ \ and\ \ }F_0=\rho^*(0)\cong \A^1\cup_{\{0\}}r\A^1.$$
The projection $p=\pr_u|_{S(k,\bar rk,1)}\:S(k,\bar rk,1)\to \A^1$ is an $\A^1$-fibration with a unique degenerate fiber $p^*(0)\cong k\A^1$.  By Theorem \ref{thm:GJC_fails_v2} some lift of $\eta$ is given by the formula
\begin{equation*}
\ti \eta(x,y,z)=(\lambda xz^{1-\alpha}R_2(1-z^k),\lambda^{-r}yR_0(1-z^k),z^{\alpha}R_1(1-z^k)),
\end{equation*}
hence we can express $\eta\in \Etc(S(k,\bar rk,1))$ as:
\begin{equation}\label{eq:eta_for_S(k,kr,1)}
\eta(u,v,w)=(u(1-t)^{1-\alpha}\lambda^kR^k_2(t),v\lambda^{-\bar r k}R_0(t),\lambda wR_1(t)R_2(t)),
\end{equation}
where $t=-u^{\bar r}v$. By Lemma \ref{lem:Ri-condition} the induced morphism on the base of $\rho$ is given by the formula
\begin{equation}\label{eq:eta_rho,k|r}
\eta_\rho(t)=t(1-t)^{(1-\alpha)\bar r}R_0(t)R_2^r(t)=1-(1-t)^\alpha R_1^k(t).
\end{equation}

\med If $k\mid \deg \eta$ then we get a surprising factorization.

\begin{lem}[Factorization through the universal cover]\label{lem:k|r} Let $\eta\in \Etc(S(k,r,a))$. If $k\mid\deg \eta$ then $\bar r:=r/k\in \N$, $a=1$ and $\eta$ factorizes $\C^*$-equivariantly through the universal covering morphism, i.e.\ there exists a $\C^*$-equivariant étale morphism $j_\eta\:S(k,\bar r k,1)\to \ti S(k,\bar r k)$ of degree $\deg j_\eta\equiv \bar r\mod (r-1)$
such that $\eta=\pi\circ j_\eta$.
\end{lem}

\begin{proof} By Theorem \ref{thm:GJC_fails_v2} we have $\alpha=0$, $r=\bar r k$ for some positive integer $\bar r$, and $a=1$. Moreover, $\eta$ is given by the formula \eqref{eq:eta_for_S(k,kr,1)}. Put 
\begin{equation}\label{eq:eta_S(k,rbark,1)} 
j_\eta(u,v,w)=(w\lambda R_2(t),v\lambda^{-r}R_0(t),R_1(t))\in \Spec (\C[x,y,z]),
\end{equation} where $t=-u^{\bar r}v$. Since on $S(k,\bar r k,1)$
$$R_1^k(t)-1=-t(1-t)^{\bar r}R_0(t)R_2^r(t)=u^{\bar r}v(1+u^{\bar r}v)^{\bar r}R_2^r(t)R_0(t)=w^rR_2^r(t) vR_0(t),$$ 
we see that $j_\eta(S(k,\bar rk,1))\subseteq \ti S(k,\bar rk)$. We check that $\pi\circ j_\eta=\eta$. Since $\eta$ is étale, $j_\eta$ is étale. 
\end{proof}

Note that since the universal cover $\pi\:\ti S(k,\bar rk)\to S(k,\bar rk,1)$ has degree $k$, the condition $k\mid\deg \eta$ is necessary for the existence of a factorization of $\eta$ through $\ti S$. Note also that if $\deg j_\eta=1$ then $d=k$, so by \eqref{eq:Ri-degrees} $\alpha=0$ and $\bar r=1$. This case is of particular interest. Put
\begin{equation*}
\ti S_k=S(k,k) \text{\ \ and\ \ }S_k=S(k,k,1).
\end{equation*}

\begin{cor}[$S_k$ embeds into $\ti S_k$]\label{cor:embedding_S_into_wtS} For every $k$-th root of unity  $\eps\neq 1$ the morphism
\begin{equation}
j(u,v,w)=(w,vR_0(-uv),R_1(-uv)),
\end{equation}
where $R_1(t)=(\eps-1)t+1$ and $R_0(t)=\frac{((\eps-1)t+1)^k-1}{t(t-1)}\in \C[t]$, defines an open $\C^*$-equivariant embedding $S_k=\ti S_k/\Z_k\mono \ti S_k.$ Different choices of $\eps\in \Z_k\setminus\{1\}$ lead to embeddings which are not related by a composition with the $\Z_k$-action on $\ti S_k$.
\end{cor}

\begin{proof} Put $\alpha=0$. By Theorem \ref{thm:GJC_fails_v2} there exists $\eta\in \Etc(S_k)$ of degree $k$. By Lemma \ref{lem:k|r} it factorizes through a $\C^*$-equivariant embedding into $\ti S_k$ given by \eqref{eq:eta_S(k,rbark,1)}. Composing with the action of $\lambda^{-1}$ we may assume $\lambda=1$. By \eqref{eq:Ri-degrees} $\deg R_2=0$ (hence $R_2=1$), $\deg R_1=1$ and $\deg R_0=k-2$. By \eqref{eq:eta_rho,k|r} $\eta_\rho(t)=t(1-t)R_0(t)=1-R_1^k(t)$. In particular, $\eta_\rho(1)=0=1-R_1^k(1)$, so since $R_1(0)=1$, we get $R_1(t)=(\eps-1)t+1$ for some $k$-th root of unity $\eps\neq 1$ and hence $R_0(t)=\frac{((\eps-1)t+1)^k-1}{t(t-1)}\in \C[t]$. Since the first coordinate of $j$ does not depend on $\eps$, we infer that different choices of $\eps\in \Z_k$ lead to embeddings which are not related by the $\Z_k$-action on $\ti S_k$.
\end{proof}

\begin{example}(Explicit embedding of $S_2$ into $\ti S_2$). For $k=2$ we have $$S_2\cong \{(u,v,w):u(1+uv)=w^2\},$$ $R_1(t)=-2t+1$ and $R_0(t)=4$, hence  $j(u,v,w)=(w,4v,1+2uv)$. Then $\eta=\pi\circ j$ is given by
\begin{equation}
\eta(u,v,w)=(u(1+uv),4v,w(1+2uv)).
\end{equation}
\end{example}

\medskip

\medskip \subsection{Deformations of étale endomorphisms.}\label{ssec:deformations}
 In this section we prove Theorem \ref{thm:deformations}. We denote by $\A^\8$ the infinite affine space defined as the colimit of open immersions $\id_{\A^n}\times \{0\}\:\A^n\mono \A^n\times \A^1=\A^{n+1}$, $n\geq 1$.
 \begin{notation}
 Given a variety $S$ the automorphism group $\mathrm{Aut}(S)$ of $S$ acts on the monoid of étale endomorphisms $\Et(S)$ by left and right compositions. We denote the set of double cosets ${\mathrm{Aut}(S)\setminus}\!\Et(S)\!{/\mathrm{Aut}(S)}$ of these actions by $\EC(S)$. 
 \end{notation}
 Recall that for every polynomial $P\in\C[x]$ the surface $\ti S(k,r)$ has an automorphism $\Theta^{P}:=\Theta^{P}_1$ (see Example \ref{ex:tiS(k,r)}) given by 
 \begin{equation}\label{eq:Theta_P} \Theta^{P}(x,y,z):=\Theta^{P}_1(x,y,z)=(x,y+x^{-r}((z+P(x) x^r)^k-z^k), z+P(x)x^r). 
 \end{equation}
Combining it with the factorization obtained in Corollary \ref{cor:embedding_S_into_wtS} we obtain the following result, which establishes Theorem \ref{thm:deformations}. 
\begin{prop}[Family of étale endomorphisms]\label{prop:deformations} 
Let $S=S(k,\bar r k,1)\cong \{(u,v,w):u(1+u^{\bar r}v)=w^k\}\subseteq \Spec (\C[u,v,w])$, $k\geq 2$, $\bar r\geq 1$ and let $\eta \in \Etc(S)$ be such that $k\mid\deg \eta$. Write $\eta=\pi\circ j_\eta$ for $j_\eta\:S\to\ti S(k,\bar rk)$ as in Lemma \ref{lem:k|r}. Then the map $\Omega_\eta\:\A^\8\to \EC(S)$ defined by 
\begin{equation}
\Omega_\eta(\mathbf{a}):=[\pi\circ \Theta^{F(\mathbf{a})}\circ j_\eta], 
\end{equation}
where $F(\mathbf{a})=F(a_1,\ldots,a_n)=1+\sum_{i=1}^n a_{i}x^{ri}\in \C[x^r]$ and $r=\bar r k$, is injective.
 
In particular, for every $N\geq 0$ there exist arbitrarily high-dimensional families of étale endomorphism of the pseudo-plane $S$ of degree $k(N(\bar r k-1)+\bar r)$ whose members are different even after dividing by the action of $\Aut(S)$ by left and right compositions. 
\end{prop}

\begin{proof} By Theorem \ref{thm:GJC_fails_v2} for every integer $N\geq 0$ the surface $S=S(k,\bar rk,1)$ has a $\C^*$-equivariant étale endomorphism of degree $Nk(\bar rk-1)+\bar rk$, so we only need to prove the injectivity of $\Omega_\eta$. Let $\mathbf{a}_i=(a_{i,1},\ldots, a_{i,n_i})\in \A^\8$, $i=1,2$ be such that the étale classes of $\eta_i:=\pi\circ \Theta^{F(\mathbf{a}_i)}\circ j_\eta$ in $\EC(S)$  are the same, i.e.\ that there exist automorphisms $\alpha_i\in \Aut(S)$, $i=1,2$, such that $\alpha_1\circ \eta_1=\eta_2\circ \alpha_2$. Set $F_i=F(\mathbf{a}_i)\in \C[x]$, $i=1,2$. Put $t=-u^{\bar r}v$. For some $\lambda\in \C^*$ we have $j_\eta=j\circ \Lambda_\lambda$, where $\Lambda_\lambda(u,v,w)=\lambda\cdot (u,v,w)$ and  
\begin{equation}
j(u,v,w)=(wR_2(t),vR_0(t),R_1(t)),
\end{equation}
hence $$\alpha_1\circ\pi\circ \Theta^{F_1}\circ j\circ \Lambda_\lambda=\pi\circ \Theta^{F_2}\circ j\circ \Lambda_\lambda\circ \alpha_2.$$  Composing both sides with $\Lambda_{\lambda^{-1}}$ and replacing $\alpha_2$ with $\Lambda_\lambda\circ \alpha_2 \circ \Lambda_{\lambda^{-1}}$, we may assume that $\lambda=1$ and hence that $$\alpha_1\circ\pi\circ \Theta^{F_1}\circ j=\pi\circ \Theta^{F_2}\circ j\circ  \alpha_2.$$
 
Let $[m]_k$ denote taking the integer $m$ modulo $k$, that is $[m]_k\in\{0,1,\ldots,k-1\}$ and $k|(m-[m]_k)$. By \cite[4.7]{MiyMa-hp_with_torus_actions} every element of $\Aut (S(k,r,a))$ has a lift in $\Aut(\wt S(k,r))$ which can be written as a composition $\beta_{J_0}\circ \beta_H$, where $\beta_H(x,y,z)=(\lambda x,\lambda^{-r}y,\bar \eps z)$ for some $\lambda\in\C^*$ and some $k$-th root of unity $\bar \eps$ and $\beta_{J_0}(x,y,z)=(x,\ldots,z+x^{r+[-a-r]_k}g(x^k))$ for some $g\in \C[x]$ (the second coordinate can be computed from the equation of $\ti S(k,r)$, so we skip the precise formula, because we will not need it). Since in our case $r=\bar rk$, we have $\beta_{J_0}(x,y,z)=(x,\ldots, z+x^{r+k-1}g(x^k))$. This implies in turn that every automorphism of $S$ has the form  
\begin{equation} 
\alpha_{\lambda,Q}(u,v,w)=(\lambda^k u,\ldots,\lambda \bar \eps w+u^{\bar r +1}Q(u)) 
 \end{equation}
 for some $\lambda \in \C^*$, some $k$-th root of unity $\bar \eps$ and some polynomial $Q\in \C[u]$. Replacing $\lambda$ with $\lambda/\bar \eps$ we see that without loss of generality we may assume $\bar \eps=1$. For both $\alpha_i$ we can therefore write $\alpha_i=\alpha_{\lambda_i,Q_i}$, for some $\lambda_i\in \C^*$ and $Q_i\in \C[u]$, hence  $$\alpha_i^*u=\lambda_i^ku \text{\ \ and \ \ } \alpha_i^*w=\lambda_i w+u^{\bar r+1}Q_i(u).$$

We have $t=-u^{\bar r}v=1-w^k/u\in \C[ S]$ and we set $T=wR_2(t)\in \C[ S]$. From now on, we view the coordinate ring $\C[ S]=\C[u,v,w]/(u(1+u^{\bar r}v)-w^k)$ of $S$  as a subring of $$\C[u^{\pm 1},v,w]/(u(1+u^{\bar r}v)-w^k)\cong \C[u^{\pm 1},w].$$ Note that $T$ is transcendental over $\C$, as $R_2\neq 0$.
 Since  $\eta_i^*=j^*(\Theta^{F_i})^*\pi^*$, we obtain $$\eta_i^*u=T^k \text{\ \ and \ \ } \eta_i^*w=R_1(t)T+T^{r+1}F_i(T).$$ 
 Applying the identity 
 \begin{equation}\label{eq:similarlity_of_eta_1_and_eta_2}
  \eta_1^*\circ \alpha_{\lambda_1,Q_1}^*=\alpha_{\lambda_2,Q_2}^*\circ \eta_2^*, 
  \end{equation}
   to the function $u\in \C[S]$ gives $$\lambda_1^k T^k=(\lambda_2w+u^{\bar r+1}Q_2(u))R_2(t')^k\in \C[u^{\pm 1},w],$$ where $t'=\alpha_2^*t=1-(\alpha_2^*w)^k/(\lambda_2^ku)$. Putting $w=0$ we get $0=u^{\bar r+1}Q_2(u)R_2(1)^k$. Since $R_2(1)\neq 0$ by the definition of $R_2$ (see Lemma \ref{lem:Ri-condition}) it follows that $Q_2=0$. So $\alpha_2^*w=\lambda_2 w$, from which it follows that $t'=\alpha^*_2 t=t$. So the above equality implies that $\lambda_1 T=\eps \lambda_2wR_2(t')=\eps\lambda_2 T$ for some $k$-th root of unity $\eps$, which yields in turn that  $\lambda_1=\eps \lambda_2$. 
 Using the above expressions for $\alpha_1^*w$ and $\eta_2^*w$ we compute that  $$\eta_1^*\circ \alpha_1^*w=\lambda_2\eps R_1(t)T+ \lambda_2\eps T^{r+1}F_1(T)+ T^{k+r}Q_1(T^k)$$ and  $$\alpha_2^*\circ \eta_2^*w=\lambda_2 R_1(t)T+\lambda_2^{r+1}T^{r+1}F_2(\lambda_2 T).$$ Plugging this into \eqref{eq:similarlity_of_eta_1_and_eta_2} and dividing both sides by $T$ we obtain $$\lambda_2(\eps-1)R_1(t)+ \lambda_2\eps T^r F_1(T)+ T^{k+r-1}Q_1(T^k)=\lambda_2^{r+1}T^r F_2(\lambda_2T).$$ 
 For $w=0$ we get $\lambda_2(\eps-1)R_1(1)=0$, hence $\eps=1$, as $R_1(1)\neq 0$. Dividing the previous equality by $T^r$ we find the relation $$\lambda_2 F_1(T)+ T^{k-1}Q_1(T^k)=\lambda_2^{r+1}F_2(\lambda_2T)\in \C[u^{\pm 1},w].$$ 

Note that so far the argument works with any formula for $F$, provided $F(\mathbf{a})\in \C[x]$. Since by hypothesis $k\geq 2$ and $F_i\in \C[x^r]\subset \C[x^k]$, we get that $Q_1=0$ and 
\begin{equation}\label{eq:similarity_in_terms_of_F}
F_1(x)=\lambda_2^rF_2(\lambda_2 x).
\end{equation}
By the choice of $F_i$ we have $F_i(0)=1$, so $F_1(x)=F_2(\lambda_2 x)$ for some $r$-th root of unity $\lambda_2$. Since $F_i\in \C[x^r]$, we obtain $F_1=F_2$ and hence $\mathbf{a}_1=\mathbf{a}_2$.
 \end{proof}
 
\begin{rem}[Lifting the family]\label{rem:lifts_of_deformations} 
Let $S=S(k,\bar r k,1)$, $\ti S=\ti S(k,\bar rk)$. As above, having a polynomial $P\in \C[x]$ we have $\eta^P=\pi\circ \Theta^P\circ j\in \Et(S)$. Put $\ti \eta=j\circ \pi\in \Et(\ti S)$. Then $\pi \circ (\Theta^P\circ \ti \eta)=\pi \circ \Theta^P\circ j\circ \pi=\eta^P\circ\pi$, so $\Theta^P\circ \ti \eta$ lifts $\eta^P$. In particular, $\ti \eta$ lifts $\eta$. Thus, even if the classes of $\eta^P$ and $\eta$ are different in $\EC(S)$, as it is in Proposition \ref{prop:deformations}, the classes of their lifts are equal in $\EC(\ti S)$. 
\end{rem}
 
\begin{example}[Formulas for $S_2$]\label{ex:S_2-defromations_explicit} 
Let $\ti S_2=S(2,2)$, $S_2=S(2,2,1)$ (see \eqref{eq:S(k,rk,1)}) and let $F(\mathbf{a})$ be as in Proposition \ref{prop:deformations}.  We have  
\begin{eqnarray*} 
j(u,v,w)&=&(w,4v,1+2uv),\\ \Theta^P(x,y,z)&=&(x,y+2zP(x)+x^2P^2(x),z+x^2P(x))),\\ \pi(x,y,z)&=&(x^2,y,xz),\\ F(\mathbf{a})&=&1+x^2Q_\mathbf{a}(x^2),  
\end{eqnarray*} 
where $\mathbf{a}=(a_1,\ldots,a_n)\in \A^\8$ and $Q_\mathbf{a}(x)=a_1+a_2x+\ldots+a_nx^{n-1}$. 
 We obtain  $$\eta^{\mathbf{a}}:=\pi\circ\Theta^{F(\mathbf{a})}\circ j=(\eta_1,\eta_2,\eta_3),$$ where 
\begin{eqnarray*} 
\eta_1&=&w^2,\\ \eta_2&=&4v+2(1+2uv)(1+w^2Q_{\mathbf{a}}(w^2))+w^2(1+w^2Q_{\mathbf{a}}(w^2))^2,\\ \eta_3&=&(1+2uv)w+w^3(1+w^2Q_{\mathbf{a}}(w^2)).  
\end{eqnarray*}
By construction, all these étale endomorphisms of $S_2$ have degree $2$, and by Proposition \ref{prop:deformations} for different $a\in \A^{\8}$ no two of them are related by compositions by automorphisms of $S_2$ from left and right. 
\end{example}

\begin{rem}[Deforming $\eta$]
Set $S=S(k,r,1)$, where, as above, $r=\bar r k$, $k\geq 2$ and $\bar r\geq 1$. Note that the family of étale endomorphisms we constructed, $\{\eta^{\mathbf{a}}\}_{\mathbf{a}\in\A^\8}\subseteq \Et(S)$, where $\eta^{\mathbf{a}}=\pi\circ \Theta^{F(\mathbf{a})}\circ j_\eta$, does not contain any $\C^*$-equivariant member, in particular it does not contain the initial $\eta$. Indeed, if $\eta^{\mathbf{a}}$ is $\C^*$-equivariant for some $\mathbf{a}\in\A^\8$ then the condition \eqref{eq:similarlity_of_eta_1_and_eta_2} holds with $\eta_1=\eta_2=\eta^{\mathbf{a}}$ and every $\lambda_1=\lambda_2=\lambda\in \C^*$, so \eqref{eq:similarity_in_terms_of_F} holds with $F_1=F_2$ for every $\lambda\in\C^*$. This is impossible, because by definition $F(0)=1$.

Still, we can embed the initial $\eta\in \Etc(S)$ in a non-trivial family with members pairwise distinct in $\EC(S)$. For instance, take
\begin{equation}
\hat \eta^{a}:=\pi\circ\Theta^{P(a)}\circ j_\eta,\ \ a\in \A^1,
\end{equation}
where $P(a)=a^2+ax^r$. Let $a,b\in \A^1$. By the above proof, cf.\ \eqref{eq:similarity_in_terms_of_F}, we have 
\begin{align*}
[\hat \eta^a]=[\hat \eta^b] \text{\ in\ } \EC(S) & \iff \exists\lambda\in \C^*\: b^2+bx^r=\lambda^r (a^2+a(\lambda x)^r) \\
&\iff \exists\lambda\in \C^*\: b=\lambda^{2r}a \text{\ \ and \ } b^2=\lambda^r a^2 \\ &\iff \exists \lambda\in \C^*\: b=\lambda^{2r} a \text{\ \ and \ }  \lambda^{3r}=1\\ &\iff \exists t\in \C^*\: tb=a \text{\ \ and \ } t^{3}=1.
\end{align*}
Thus, changing the parameterizing variety to $\A^1/\Z_3$, we get a family as required.
\end{rem}

\bibliographystyle{amsalpha} 
\bibliography{bibl2016}
\end{document}